\documentclass[12pt]{amsart}
\usepackage{graphicx}
\usepackage{graphicx}
\usepackage{pstricks}
\usepackage{epsf}

\newtheorem{thm}{Theorem}[section]

\newtheorem{defn}[thm]{Definition}
\newtheorem{lemma}[thm]{Lemma}

\newtheorem{remark}[thm]{Remark}
\newtheorem{example}[thm]{Example}

\usepackage{amsmath}
\usepackage{amsxtra}
\usepackage{amscd}
\usepackage{amsthm}
\usepackage{amsfonts}
\usepackage{amssymb}
\usepackage{eucal}
\newcommand{\bmb}{\left( \begin{array}{rr}}
\newcommand{\enm}{\end{array}\right)}
\newcommand{\Z}{{\mathbb Z}}

\newcommand{\bm}{{\mathbf m}}

\newcommand{\bp}{{\mathbf p}}

\newcommand{\al}{{\alpha}}

\numberwithin{equation}{section}

\begin{document}
\title{The Non-Commutative $A_1$ $T$-system and its positive Laurent property}
\author{Philippe Di Francesco} 
\address{
Department of Mathematics, University of Illinois\\
Urbana, IL 61821 USA. \\
e-mail: philippe@illinois.edu
}
\begin{abstract}
We define a non-commutative version of the $A_1$ T-system, which underlies frieze patterns of the integer plane.
This system has discrete conserved quantities and has a particular reduction to the known non-commutative Q-system for $A_1$. 
We solve the system by generalizing the flat $GL_2$ connection method used in the commuting case to a 2$\times$2 flat matrix connection 
with non-commutative entries. This allows to prove the non-commutative positive Laurent phenomenon for the solutions when expressed in terms
of admissible initial data. These are rephrased as partition functions of paths with non-commutative weights on networks, and 
alternatively of dimer configurations with non-commutative weights on ladder graphs made of chains of squares and hexagons.
\end{abstract}

\maketitle
\date{\today}
\tableofcontents

%\section{Introduction}

\section{The $A_1$ $T$-system and its initial data}

$T$-systems were defined in the context of functional relations satisfied by the transfer matrices of integrable quantum spin chains with a Lie group symmetry \cite{KNS}.
Together with their associated $Y$-systems, these are instrumental in the Bethe Ansatz solution of these quantum systems. There is such a system for each classical Lie group, and in this note we will concentrate on the simplest $A_1$ case.

The $A_1$ $T$-system is the following system of recursion relations
\begin{equation}\label{clatsys}
T_{j,k+1}T_{j,k-1}=1+T_{j+1,k}T_{j-1,k}\qquad (j,k \in \Z)
\end{equation}
This expresses the ``frieze" condition \cite{Cox,FRISES} that on each elementary square of the lattice $\Z^2$ (tilted by $45^\circ$), 
the determinant of the corner variables is 1:
$$\det\begin{pmatrix} T_{j,k-1} & T_{j-1,k} \\T_{j+1,k}&T_{j,k+1} \end{pmatrix}=1$$

We restrict to the system \eqref{clatsys} with $j+k=1$ mod 2, in other words we only consider variables $T_{j,k}$ with $j+k=0$ mod 2.
Admissible initial data for this system are attached to a pair $(\bm,x_\bm)$ where $\bm$ is an infinite path
$\bm=(m_j)_{j\in\Z}$ such that $m_j\in\Z$, $j+m_j=0$ mod 2, 
and $|m_{j+1}-m_j|=1$ for all $j\in\Z$, and $x_\bm=\{t_j\}_{j\in \Z}$ is an infinite sequence of initial values along the path $\bm$.
The initial condition is simply the assignments:
\begin{equation}\label{ass} T_{j,m_j}=t_j, \qquad (j\in \Z) \end{equation}

The fundamental initial data corresponds to the path $\bm^{(0)}$,
with $m_j^{(0)}=j\,{\rm mod}\, 2$ for all $j\in\Z$.
Any other initial data $\bm$ may be obtained from $\bm_0$ via (forward/backward) mutations 
of the form $\mu_\ell^{\pm}$, $\ell\in \Z$, such that $\bm'=\mu_\ell^{\epsilon}(\bm)$ iff
$m'_j=m_j+ 2\epsilon \delta_{j,\ell}$, $\epsilon\in\{-1,1\}$, both $\bm$ and
$\bm'$ being paths. Each such mutation leaves the data
$t_{j}=t_{j}'$ unchanged for $j\neq \ell$, and updates
$t_\ell=T_{\ell,m_\ell}\to {t_\ell}'=T_{\ell,{m_\ell}'}$ by use of the relation \eqref{clatsys} for $j=\ell$, $k=m_\ell+\epsilon={m_\ell}'-\epsilon$,
namely ${t_\ell}'=(1+t_{\ell-1}t_{\ell+1})/t_\ell$.

This system was extensively studied and solved for various boundary conditions \cite{DFK13}. 
One important feature of the system is that it may be interpreted as attached to some special mutations in a suitable infinite rank cluster algebra
\cite{DFK08}. As such the solution is expected to be expressible as a Laurent polynomial of any admissible initial data (due to the Laurent property of cluster algebra \cite{FZI}). The solutions show that this Laurent phenomenon produces
Laurent polynomials with only non-negative integer coefficients. In addition, the system was found to be integrable, in that it admits two infinite sequences of conserved quantities \cite{DFK11}.
The connection to cluster algebra has allowed to define a quantum version of the system, 
via the associated quantum cluster algebra \cite{BZ},
in which variables no longer commute, but are subject to $q$-commutation relations within 
each cluster. It was shown in \cite{DFK11} that an analogous quantum Laurent property holds, 
this time with coefficients in $\Z_+[q,q^{-1}]$.

In all cases, the system was solved by introducing a flat $GL(2)$ connection, namely a $2\times 2$ matrix representation of the $T$-system relation, allowing for writing a  compact formula for the solution,  The latter may be interpreted in terms of paths on a network, or equivalently of matchings or dimer coverings of some suitable graphs \cite{DFK11,DFK13,DF13}.

All the $T$-systems admit a reduction to so-called $Q$-systems, in which the ``spectral parameter" index $j$ is omitted. For the case of $A_1$, this simply reads: $R_{k+1}R_{k-1}=R_{k}^2+1$.
One way of thinking of $R_{k}$ is as a solution $T_{j,k}$, $j+k=0$ mod 2, of the $T$-system \eqref{clatsys} which is 2-periodic in the variable $j$, 
for the 2-periodic fundamental initial data $(\bm^{(0)},x_{\bm^{(0)}})$ with the flat path 
$\bm^{(0)}$ defined above and assignments $(t_j)_{j\in \Z}$ such that $t_{j+2}=t_j$ for all $j$. This reduces effectively to the initial conditions $R_0=t_0$ and $R_1=t_1$ for the corresponding $Q$-system. 
This system is naturally attached to a rank 2 cluster algebra.

In addition to a quantum version inherited from the quantum version of the associated cluster algebra \cite{BZ}, 
this system was extended to a fully non-commutative setting \cite{Kon,DFK09b}, in which 
$R_k$ are now non-commuting invertible elements of a unital algebra $\mathcal A$. 
We still impose initial conditions $R_0=t_0$ and $R_1=t_1$, with fixed invertible elements 
$t_0,t_1\in {\mathcal A}$. This system is still integrable, i.e. 
it admits conserved quantities, and its solutions satisfy the non-commutative Laurent property.
It was solved using a formulation of $R_k$ as a partition function of paths with non-commutative weights, monomial in $t_0^{\pm1},t_1^{\pm 1}$
\cite{DFK09b}, thus displaying positivity of the coefficients as well.
This system was used to generalize cluster algebras for non-commutative cluster variables in rank 2, and a few examples in higher rank were constructed as well \cite{DFK10}. The Laurent property, and finally its positive version, were proved for non-commutative rank 2 cluster algebras in 
\cite{BR,LS,RUP}, thus completing the proof of a general conjecture by Kontsevich.

The purpose of this note is to introduce and solve a fully non-commutative version of the $A_1$ $T$-system, 
still enjoying the positive Laurent property in terms of non-commuting initial variables, and to solve it with 
non-commutative extensions of the previous solutions. The paper is organized as follows.

In Section 2, we define the non-commutative $A_1$ $T$-system and show its invariance and integrability properties.
The novel feature is that it involves {\it two} sets of non-commuting invertible variables of a unital algebra $\mathcal A$
at each site, subject to local mixed relations.

In Section 3 we show that this system reduces to respectively the non-commutative $A_1$ $Q$-system and the quantum $A_1$ $T$-system
for suitable choices of the variables.

In Section 4 we construct the general solution of the system, by introducing a flat $GL_2({\mathcal A})$ connection defined on its solutions.
This extends the method already used successfully in the commuting and quantum cases, and produces an explicit formula for the solution,
which displays the positive Laurent property manifestly. We further interpret this solution in the language of (non-commutative) networks, namely models of paths on directed graphs with non-commuting weights. The latter are identified bijectively with configurations of dimers on some suitable ladder graphs with non-commuting weights involving the initial data for the system.

We gather a few concluding remarks in Section 5.

\medskip

\noindent{\bf Acknowledgments.}  We thank M. Gekhtman, R. Kedem and M. Shapiro for discussions, and especially 
V. Retakh for an inspiring talk and subsequent discussions at the Oberwolfach workshop ``Cluster algebras and 
related topics", December 7-14, 2013 (see also \cite{RFO}). 
This work is supported by the NSF grant DMS 13-01636 and 
the Morris and Gertrude Fine endowment.

\section{The non-commutative  $A_1$ $T$-system}

\subsection{Definition}

We consider formal invertible, non-commuting variables $T_{j,k}$, $j,k\in \Z$, $j+k=0$ mod 2, and an involutive anti-automorphism denoted
by $\!\bullet$, providing us with a second collection of independent variables $T_{j,k}^\bullet$.

\begin{defn}\label{defaone}
The non-commutative (NC) $A_1$ $T$-system is the following system of recursion relations:
\begin{equation}\label{ncaonet}
T_{j,k+1} T_{j,k-1}^\bullet=1+T_{j-1,k} T_{j+1,k}^\bullet\qquad (j,k\in \Z)\end{equation}
together with the relations
\begin{equation}\label{quasico} T_{j,k-1}^{-1}T_{j+1,k}=T_{j+1,k}^\bullet(T_{j,k-1}^\bullet)^{-1}\quad {\rm and}\quad 
T_{j-1,k}T_{j,k-1}^{-1}=(T_{j,k-1}^\bullet)^{-1}T_{j-1,k}^\bullet \end{equation}
for all $j,k\in \Z$, $j+k=1$ mod 2.
\end{defn}

The system \eqref{ncaonet} above may be equivalently rewritten as:
$$ T_{j,k+1}= (T_{j,k-1}^\bullet)^{-1}+T_{j-1,k} T_{j,k-1}^{-1}T_{j+1,k}\qquad (j,k\in \Z; j+k=1\, {\rm mod}\, 2)$$
while \eqref{quasico} is unchanged. In terms of quasideterminants, this is also equivalent to:
$$ T_{j,k-1}^\bullet\, \left\vert \begin{matrix}T_{j,k-1} & T_{j-1,k}\\ T_{j+1,k}& T_{j,k+1}\end{matrix}\right\vert_{2,2}=1$$
in the notations of  \cite{GGRW}.

The system \eqref{ncaonet} must be supplemented by some initial conditions  defined as follows. 
Admissible initial data are now coded by a triple $(\bm,x_\bm,x_\bm^\bullet)$ 
made of an infinite path $\bm=(m_j)_{j\in \Z}$ (as in the commuting case)
and two sequences $x_\bm=\{t_j\}_{j\in \Z}$ and $x_\bm^\bullet=\{t_j^\bullet\}_{j\in \Z}$ of initial values, 
satisfying the following nearest neighbor relations:
\begin{eqnarray}
t_j^{-1}\, t_{j+1}&=&t_{j+1}^\bullet\, (t_j^\bullet)^{-1}  \quad {\rm if}\, m_{j+1}=m_j+1 \nonumber \\
t_j\, t_{j+1}^{-1}&=&(t_{j+1}^\bullet)^{-1}\, t_j^\bullet \quad  {\rm if}\, m_{j+1}=m_j-1  \label{quasicoinit}
\end{eqnarray}
The initial condition  $(\bm,x_\bm,x_\bm^\bullet)$ is simply the assignments:
\begin{equation}\label{ncass} 
T_{j,m_j}=t_j \qquad {\rm and} \qquad T_{j,m_j}^\bullet=t_j^\bullet 
\end{equation}

\subsection{Reflection invariance}

It is useful to note that the system is invariant under the composition of $\bullet$ and the symmetry $S_{a,b}$ defined
as $ST_{j,k}=T_{a-j,b-k}$ for some fixed integers $a,b$ such that $a+b=j+k$ mod 2. We have:
\begin{lemma}\label{reflec}
If $T_{j,k}$ is a solution of the NC $A_1$ $T$-system, then $S_{a,b}T_{j,k}=T_{a-j,b-k}$ is also a solution of the same system.
Moreover if we set initial conditions of the form $(\bm,x_\bm,x_\bm^\bullet)$, then
the quantity $S_{a,b}T_{j,k}=T_{a-j,b-k}$ solves the NC $A_1$ $T$-system for the initial data $(\bp,x_\bp,x_\bp^\bullet)$,
where $\bp=S_{a,b}\bm$ is the reflected path with $p_j=b-m_{a-i}$, $j\in \Z$, and the data $x_\bp=\{t_{a-j}\}_{j\in \Z}$,
$x_\bp^\bullet=\{t_{a-j}^\bullet\}_{j\in \Z}$.
\end{lemma}
\begin{proof}
Let us apply $\bullet$ to the system \eqref{ncaonet} and then take $k\to b-k$ and $j\to a-j$. We get
\begin{equation}
T_{a-j,b-k-1} T_{a-j,b-k+1}^\bullet=1+T_{a-j+1,b-k} T_{a-j-1,b-k}^\bullet
\end{equation}
Let us define $U_{j,k}=S_{a,b}T_{j,k}=T_{a-j,b-k}$. Then the above equation becomes:
$$ U_{j,k+1} U_{j,k-1}^\bullet=1+U_{j-1,k} U_{j+1,k}^\bullet$$
and the first statement of the lemma follows. The second follows from noting that the initial data for $U_{j,k}$ are exactly 
given by $(\bp,x_\bp,x_\bp^\bullet)$.
\end{proof}

If we choose for $(a,b)$ some vertex of the initial data path, Lemma \ref{reflec} allows us to only
compute the solution of the $T$-system {\it above} the initial data path, namely $T_{j,k}$ for all $k\geq m_j$.
Indeed, upon applying the reflection w.r.t. a vertex of $\bm$, we immediately get the solution {\it under} the path.
Throughout these notes we will therefore restrict ourselves to finding the solution above the path.

\subsection{Integrability}

The NC $A_1$ $T$-system is integrable in the following sense.
\begin{thm}\label{consqthm}
The quantities
\begin{eqnarray*}\Gamma_{j,k}&=& T_{j-1,k+1}T_{j,k}^{-1}+(T_{j,k}^\bullet)^{-1}T_{j+1,k-1}^\bullet \\
\Delta_{j,k}&=&T_{j,k}^{-1}T_{j+1,k+1}+T_{j-1,k-1}^\bullet (T_{j,k}^\bullet)^{-1}
\end{eqnarray*}
are conserved modulo the NC $A_1$ $T$-system, respectively in the $(1,1)$ and $(-1,1)$ directions,
namely: $\Gamma_{j,k}=\Gamma_{j-1,k-1}=\Gamma_{j-k,0}\equiv \Gamma_{j-k}$ and 
$\Delta_{j,k}=\Delta_{j+1,k-1}=\Delta_{j+k,0}\equiv \Delta_{j+k}$ for all $j,k\in \Z$.
Moreover, the NC $A_1$ $T$-system solutions satisfy the following left/right linear recursion relations:
\begin{eqnarray*} T_{j-1,k+1}-\Gamma_{j-k}\, T_{j,k}+T_{j+1,k-1}&=&0 \\
T_{j+1,k+1}-T_{j,k}\, \Delta_{j+k}+T_{j-1,k-1}&=&0
\end{eqnarray*}
in which the non-trivial coefficient is a conserved quantity.
\end{thm}
\begin{proof}
We start from eq.\eqref{ncaonet}:
\begin{equation}\label{zer}
T_{j,k+1} T_{j,k-1}^\bullet=1+T_{j-1,k} T_{j+1,k}^\bullet
\end{equation}
Let us multiply the equation \eqref{zer} on the {\it right} by $(T_{j,k-1}^\bullet)^{-1} T_{j+1,k}^{-1}$. We get:
\begin{equation}\label{prem} T_{j,k+1}T_{j+1,k}^{-1}=(T_{j,k-1}^\bullet)^{-1} T_{j+1,k}^{-1}+ T_{j-1,k} T_{j,k-1}^{-1} 
\end{equation}
On the other hand, multiplying  \eqref{zer} on the {\it left} by $(T_{j-1,k}^\bullet)^{-1} T_{j,k+1}^{-1}$ yields:
$$(T_{j-1,k}^\bullet)^{-1} T_{j,k-1}^\bullet
=(T_{j-1,k}^\bullet)^{-1} T_{j,k+1}^{-1}+(T_{j,k+1}^\bullet)^{-1}T_{j+1,k}^\bullet 
$$
Substituting $(j,k)\to (j+1,k-1)$ gives
\begin{equation}\label{deuze}T_{j,k-1}^\bullet (T_{j+1,k}^\bullet)^{-1} 
=(T_{j,k-1}^\bullet)^{-1} T_{j+1,k}^{-1}+(T_{j+1,k}^\bullet)^{-1}T_{j+2,k-1}^\bullet \, ,
\end{equation}
Subtracting \eqref{prem} from \eqref{deuze} finally yields:
$$ T_{j,k-1}^\bullet (T_{j+1,k}^\bullet)^{-1}+T_{j-1,k} T_{j,k-1}^{-1} =T_{j,k+1}T_{j-1,k}^{-1}+(T_{j+1,k}^\bullet)^{-1}T_{j+2,k-1}^\bullet $$
which reads simply $\Gamma_{j,k-1}=\Gamma_{j+1,k}$. Substituting
$(T_{j,k}^\bullet)^{-1}T_{j+1,k-1}^\bullet=T_{j+1,k-1}T_{j,k}^{-1}$ in the definition of $\Gamma_{j,k}$, 
we get $\Gamma_{i,j}=(T_{j-1,k+1}+T_{j+1,k-1})T_{j,k}^{-1}$
and the linear recursion relation follows. The results for $\Delta_{j,k}$ follow analogously. Let us multiply \eqref{zer} on the left by $T_{j-1,k}^{-1}$ and on the right by $(T_{j,k-1}^\bullet)^{-1}$, and use $T_{j+1,k}^\bullet (T_{j,k-1}^\bullet)^{-1}=T_{j,k-1}^{-1}T_{j+1,k}$ to get:
\begin{equation}\label{troize}T_{j-1,k}^{-1} T_{j,k+1}=
T_{j-1,k}^{-1} (T_{j,k-1}^\bullet)^{-1}+T_{j,k-1}^{-1}T_{j+1,k} \,
\end{equation}
On the other hand, let us multiply \eqref{zer} on the left by $T_{j,k+1}^{-1}$ and on the right by $(T_{j+1,k}^\bullet)^{-1}$, 
and use $T_{j,k+1}^{-1}T_{j-1,k}=T_{j-1,k}^\bullet (T_{j,k+1}^\bullet)^{-1}$ to get:
$$T_{j,k-1}^\bullet(T_{j+1,k}^\bullet)^{-1}=T_{j,k+1}^{-1}(T_{j+1,k}^\bullet)^{-1}+T_{j-1,k}^\bullet (T_{j,k+1}^\bullet)^{-1}
$$
Substituting $(j,k)\to (j-1,k-1)$ gives
\begin{equation}\label{quatze}T_{j-1,k-2}^\bullet(T_{j,k-1}^\bullet)^{-1}=T_{j-1,k}^{-1}(T_{j,k-1}^\bullet)^{-1}
+T_{j-2,k-1}^\bullet (T_{j-1,k}^\bullet)^{-1}
\end{equation}
Subtracting \eqref{troize} from \eqref{quatze} finally yields:
$$T_{j-1,k}^{-1} T_{j,k+1}+T_{j-2,k-1}^\bullet (T_{j-1,k}^\bullet)^{-1}= T_{j,k-1}^{-1}T_{j+1,k}+T_{j-1,k-2}^\bullet(T_{j,k-1}^\bullet)^{-1}$$
which reads simply $\Delta_{j-1,k}=\Delta_{j,k-1}$. The linear recursion relation follows
by substituting $T_{j-1,k-1}^\bullet (T_{j,k}^\bullet)^{-1}=T_{j,k}^{-1}T_{j-1,k-1}$ in the definition of 
$\Delta_{j,k}$, and the theorem is proved.
\end{proof}

\begin{remark}\label{selfad}
Both conserved quantities of Theorem \ref{consqthm} are invariant under $\bullet$ as a consequence of the relations \eqref{quasico}.
This implies the following linear recursion relations for the $T^\bullet$ variables:
\begin{eqnarray*} T_{j-1,k+1}^\bullet-T_{j,k}^\bullet \, \Gamma_{j-k}+T_{j+1,k-1}^\bullet&=&0 \\
T_{j+1,k+1}^\bullet-\Delta_{j+k}T_{j,k}^\bullet +T_{j-1,k-1}^\bullet&=&0
\end{eqnarray*}
:
\end{remark}

\subsection{Main result}

We now state the main theorem of these notes:

\begin{thm}{(NC positive Laurent Property)}\label{main}
For any fixed initial conditions $(\bm,x_\bm,x_\bm^\bullet)$ the non-commutative $A_1$ $T$-system's solutions $T_{j,k},T_{j,k}^\bullet$
are non-commutative Laurent polynomials of the initial variables $t_i,t_i^\bullet$, $i\in \Z$, with non-negative integer coefficients.
\end{thm}

To prove this theorem, we will construct the solutions of the NC $A_1$ $T$-system
explicitly by means of a matrix representation (see Section 4 below). Before going into this, we show in the next section that the $A_1$ $T$-system
above restricts to known systems for particular $T$'s and $T^\bullet$'s.

\section{Restrictions to known cases}

In this section, we show that the NC $A_1$ $T$-system restricts to some known non-commutative 
systems for suitable choices of the variables $T_{j,k}$ and $T_{j,k}^\bullet$.

\subsection{The NC $A_1$ $Q$-system}

\begin{defn}{\cite{DFK09b}}\label{aoneqdef}
The Non-Commutative $A_1$ $Q$-system is the following set of recursion relations for non-commutative invertible variables $R_n$, $n\in \Z$:
\begin{equation}\label{konts}
R_{n+1}\,R_n^{-1}\,R_{n-1}=R_n+R_n^{-1} \qquad (n\in \Z)
\end{equation}
\end{defn}

The initial data $(R_0,R_1)$ determine the value of a conserved quantity $C$ such that
$$ R_{n+1}^{-1}R_n\,R_{n+1}R_n^{-1}=R_1^{-1}R_0\,R_1R_0^{-1}=C$$
which can be rewritten as the following quasi-commutation relations:
\begin{equation}\label{kontquasi}R_n\,R_{n+1}=R_{n+1}\,C\,R_n\quad {\rm or}\quad R_n^{-1}\,R_{n+1}\,C=R_{n+1}\,R_n^{-1}
\end{equation}
while the main
equation can be rewritten as
\begin{equation}\label{kontq} R_{n+1}C R_{n-1}=1+R_n^{2} \end{equation}
or equivalently $R_{n-1} R_{n+1}C=1+R_nCR_nC$.
Moreover, there is a second conserved quantity\cite{DFK09b}:
\begin{equation}\label{cons2} 
K=R_{n+1}R_n^{-1}+R_{n}^{-1}R_{n-1}=R_1R_0^{-1}+R_1^{-1}R_0^{-1}+R_1^{-1}R_0
\end{equation}
modulo the $A_1$ $Q$-system relation \eqref{konts}.
%Let us define an anti-automorphism $\bullet$ by
%$$ R_n^\bullet=R_nC \quad {\rm and}\quad C^\bullet=C^{-1}\quad 1^\bullet=1$$

We have the following:

\begin{thm}
The NC $A_1$ $T$-system of Def.\ref{defaone} reduces to the NC $A_1$ $Q$-system of Def.\ref{aoneqdef} for particular choices of 
the non-commutative variables $T_{j,k}$.
\end{thm}
\begin{proof}
Setting:
\begin{equation}\label{chgTQ} 
T_{j,k}=C^{-a_{j,k}}R_k C^{b_{j,k}}, \qquad T_{j,k}^\bullet=C^{-c_{j,k}}R_k C^{d_{j,k}}\qquad (j,k\in \Z; j+k=0\, {\rm mod}\, 2)
\end{equation}
we easily get:
\begin{eqnarray*} T_{j,k-1}^{-1}T_{j+1,k}&=&C^{-b_{j,k-1}}R_{k-1}^{-1}C^{a_{j,k-1}-a_{j+1,k}}R_{k}C^{b_{j+1,k}} \\
T_{j+1,k}^\bullet(T_{j,k-1}^\bullet)^{-1}&=& C^{-c_{j+1,k}}R_k C^{d_{j+1,k}-d_{j,k-1}}R_{k-1}^{-1} C^{c_{j,k-1}}\\
T_{j-1,k}T_{j,k-1}^{-1}&=&C^{-a_{j-1,k}}R_k C^{b_{j-1,k}-b_{j,k-1}}R_{k-1}^{-1}C^{a_{j,k-1}}\\
(T_{j,k-1}^\bullet)^{-1}T_{j-1,k}^\bullet &=&C^{-d_{j,k-1}}R_{k-1}^{-1} C^{c_{j,k-1}-c_{j-1,k}}R_kC^{d_{j-1,k}}
\end{eqnarray*}
These boil down to the quasi-commutation \eqref{kontquasi} iff the following relations are satisfied:
\begin{eqnarray}a_{j+1,k}&=&a_{j,k-1}=d_{j-1,k}-1,\qquad  b_{j+1,k}-1=c_{j,k-1}=c_{j-1,k}\nonumber \\
c_{j+1,k}&=&b_{j,k-1}=b_{j-1,k}, \qquad \qquad \quad d_{j+1,k}=d_{j,k-1}=a_{j-1,k}\label{systemabcd}
\end{eqnarray}
These recursion relations determine $a,b,c,d$ up to initial data.
Assuming these hold, let us finally express:
\begin{eqnarray*}T_{j,k+1} T_{j,k-1}^\bullet&=&C^{-a_{j,k+1}}R_{k+1} C^{b_{j,k+1}-c_{j,k-1}}R_{k-1} C^{d_{j,k-1}}\\
T_{j-1,k} T_{j+1,k}^\bullet&=&C^{-a_{j-1,k}}R_k C^{b_{j-1,k}-c_{j+1,k}}R_k C^{d_{j+1,k}}
\end{eqnarray*}
Noting that $a_{j,k+1}=a_{j-1,k}$, $d_{j+1,k}=d_{j,k-1}$, $b_{j-1,k}-c_{j+1,k}=0$, $b_{j,k+1}-c_{j,k-1}=1$
and $a_{j,k+1}=d_{j,k-1}$ as a consequence of the relations \eqref{systemabcd}, 
we conclude that:
$$
0=T_{j,k+1} T_{j,k-1}^\bullet-T_{j-1,k} T_{j+1,k}^\bullet-1=C^{-a_{j,k+1}}(R_{k+1}CR_{k-1}-R_k^2-1)C^{d_{j,k-1}}
$$
hence \eqref{ncaonet} boils down to  \eqref{kontq} up to conjugation with $C^{a_{j,k+1}}$, and the theorem follows.
\end{proof}

Let us work out explicitly the solution of the system \eqref{systemabcd}. First note that 
\begin{eqnarray*} a_{j,k}&=&a_{j-k,0}=a\left(\frac{j-k}{2}\right),\quad  b_{j,k}=b_{j+k,0}=b\left(\frac{j+k}{2}\right), \\
c_{j,k}&=&c_{j+k,0}=c\left(\frac{j+k}{2}\right), \quad d_{j,k}=d_{j-k,0}=d\left(\frac{j-k}{2}\right),
\end{eqnarray*}
as $j=k$ mod 2.
We conclude that $a(m)=d(m-1)-1$, $b(m)=c(m-1)+1$, $c(m)=b(m-1)$ and $d(m)=a(m-1)$, which is easily solved 
say for trivial initial data $a(0)=b(0)=c(0)=d(0)=0$ as:
$$ a(m)=-\Big\lfloor\frac{m+1}{2} \Big\rfloor , \quad b(m)=\Big\lfloor\frac{m+1}{2} \Big\rfloor, 
\quad c(m)=\Big\lfloor\frac{m}{2} \Big\rfloor,
\quad d(m)=-\Big\lfloor\frac{m}{2} \Big\rfloor $$
so that we finally get
\begin{equation}\label{solabcd}
a_{j,k}=-\lfloor \frac{j-k+2}{4}\rfloor ,\ b_{j,k}=\lfloor \frac{j+k+2}{4}\rfloor, \ c_{j,k}=\lfloor\frac{j+k}{4}\rfloor,\ 
d_{j,k}=-\lfloor \frac{j-k}{4}\rfloor 
\end{equation}

\begin{remark}
We note that with this choice we have the quasi-periodicity conditions:
$T_{j+4,k}=CT_{j,k}C$ and $T_{j+4,k}^\bullet=C^{-1}T_{j,k}^\bullet C^{-1}$. This is to be contrasted with the 
commuting case, for which the solutions of the $A_1$ $T$-system that are $2$-periodic in $j$ 
are the solutions of the $A_1$ $Q$-system.
\end{remark}

\begin{remark}
The conserved quantities of the NC $A_1$ $T$-system reduce respectively to:
$$ \Gamma_{j,k}= C^{-a_{j-1,k+1}}(R_{k+1}R_k^{-1}+R_{k}^{-1}R_{k-1})C^{-b_{j,k}} 
=C^{\lfloor \frac{j-k}{4}\rfloor}KC^{-\lfloor \frac{j-k}{4}\rfloor} $$
where we have used \eqref{solabcd} and the identities: $a_{j-1,k+1}=d_{j,k}$, $a_{j,k}=d_{j+1,k-1}$, 
and $b_{j-1,k+1}=b_{j,k}$ as well as $c_{j,k}=c_{j+1,k-1}$, and:
$$ \Delta_{j,k}= C^{-b_{j,k}} (R_{k}^{-1}R_{k+1} C+C R_{k-1}R_k^{-1}) C^{c_{j,k}}=C^{-\lfloor \frac{j+k+2}{4}\rfloor} K C^{\lfloor \frac{j+k}{4}\rfloor}$$
where we have used \eqref{solabcd} and the identities: $b_{j,k}=c_{j-1,k-1}+1$, $b_{j+1,k+1}=c_{j,k}+1$ as well as $a_{j,k}=a_{j+1,k+1}$
and $d_{j-1,k-1}=d_{j,k}$. We recover the fact that $\Gamma_{j,k}$ is a function of $j-k$ only, while $\Delta_{j,k}$
is a function of $j+k$ only.
\end{remark}

%\begin{remark}
%The anti-automorphism $\bullet$ acts on $R_k$ as follows:
%$$  $$
%Indeed, we simply compute 
%$$ T_{j,k}^\bullet=(C^\bullet)^{b_{j,k}}R_k^\bullet (C^\bullet)^{-a_{j,k}}=C^{-c_{j,k}}R_k C^{d_{j,k}} \Rightarrow R_k^\bullet=(C^\bullet)^{-b_{j,k}}C^{-c_{j,k}}R_k  C^{d_{j,k}}(C^\bullet)^{a_{j,k}} $$
%Noting that $b(m)-c(m)=(1-(-1)^m)/2=d(m)-a(m)$, and that $\frac{j-k}{2}$ and $\frac{j+k}{2}$ 
%have the same parity if $k$ is even and opposite if $k$ is odd, we conclude that $R_{2k+1}^\bullet= $
%\end{remark}

\subsection{The quantum $A_1$ $T$-system}

The quantum $A_1$ $T$-system was defined by use of a cluster algebra formulation of the commuting $A_1$ $T$-system, and
by considering its natural quantum version as provided by a corresponding quantum cluster algebra.
It is defined as follows:

\begin{defn}\label{qtdef}
The quantum $A_1$ $T$-system is the following system of recursion relations:
\begin{equation}\label{qt}q\,\tau_{j,k+1}\tau_{j,k-1}=1+\tau_{j+1,k}\tau_{j-1,k} \qquad (j,k\in \Z; j+k=1\, {\rm mod}\, 2)
\end{equation}
 for non-commuting variables  $\tau_{j,k}$ subject to the following $q$-commuting relations:
\begin{eqnarray} 
\tau_{i,k-1}\tau_{j,k}&=&q^{(-1)^{\lfloor\frac{|i-j|}{2}\rfloor}} \,\tau_{j,k}\tau_{i,k-1} \quad (i,j,k\in \Z;i+k=1\, {\rm mod}\, 2;j+k=0\, {\rm mod}\, 2) \label{qco}\\
\tau_{i,k}\tau_{j,k}&=&\tau_{j,k}\tau_{i,k} \qquad \qquad \qquad  \qquad(i,j,k\in \Z;i+k=j+k=0\, {\rm mod}\, 2)\nonumber
\end{eqnarray}
\end{defn}
 
\begin{thm}
The NC $A_1$ $T$-system of Def.\ref{defaone} reduces to the quantum $T$-system of Def.\ref{qtdef}
for particular choices of the non-commutative variables $T_{j,k}$.
\end{thm}
\begin{proof}
Setting
$$  T_{j,k}=q^{\al_{j,k}} \tau_{j,k}\quad {\rm and}\quad T_{j,k}^\bullet=q^{-\beta_{j,k}} \tau_{j,k}$$
and noting that
\begin{eqnarray*} T_{j,k-1}^{-1}T_{j+1,k}&=&q^{\al_{j+1,k}-\al_{j,k-1}}\tau_{j,k-1}^{-1}\tau_{j+1,k}\\
T_{j+1,k}^\bullet(T_{j,k-1}^\bullet)^{-1}&=& q^{\beta_{j,k-1}-\beta_{j+1,k}}\tau_{j+1,k}\tau_{j,k-1}^{-1}\\
T_{j-1,k}T_{j,k-1}^{-1}&=&q^{\al_{j,k-1}-\al_{j-1,k}}\tau_{j-1,k}\tau_{j,k-1}^{-1}\\
(T_{j,k-1}^\bullet)^{-1}T_{j-1,k}^\bullet &=&q^{\beta_{j,k-1}-\beta_{j-1,k}}\tau_{j,k-1}^{-1} \tau_{j-1,k}
\end{eqnarray*}
we see that the first set of commutation relations in \eqref{qco} for $i=j-1$ are satisfied iff:
$$ \al_{j+1,k}-\al_{j,k-1}=\beta_{j,k-1}-\beta_{j+1,k}+1,\quad \al_{j,k-1}-\al_{j-1,k}=\beta_{j,k-1}-\beta_{j-1,k}-1$$
while \eqref{qt} is satisfied iff:
$$ \al_{j,k+1}-\beta_{j,k-1}=1,\quad  \al_{j-1,k}-\beta_{j+1,k}=0$$
We deduce that 
\begin{eqnarray*}\al_{j+1,k}+\al_{j-1,k}-\al_{j,k+1}-\al_{j,k-1}&=&0\\
\al_{j-1,k}-\al_{j-1,k+2}+\al_{j,k+1}-\al_{j,k-1}&=&1
\end{eqnarray*}
This system has solutions, and we find in particular that:
$$ \al_{j,k}=\frac{2k+1}{4}-\frac{(j-k)^2}{8} ,\quad \beta_{j,k}= \frac{2j-1}{4}-\frac{(j-k)^2}{8}$$
fulfills all the requirements. The other $q$-commutation relations are simply further requirements on the variables $T,T^\bullet$.
\end{proof}

\section{Solution via flat NC connection}

In this section, we introduce a representation of the NC $T$-system relation via a 2x2 matrix identity with non-commuting entries, 
which can be interpreted as a flat NC connection. This is then used to write a compact expression for the solution of the system
for any initial data.

\subsection{V,U matrices}

\begin{defn}
Let $U,V$ be the following $2\times 2$ matrices with non-commutative entries\footnote{Our notation is slightly 
abusive, as for instance $V(a,b)$ actually depends a priori on the four independent variables $a,b,a^\bullet,b^\bullet$, 
and likewise for $U$. The indication $a,b$ here really stands for $a,b$ and their $\bullet$ images.}:
$$V(a,b)=\begin{pmatrix}a b^{-1} & (b^\bullet)^{-1}\\ 0 & 1\end{pmatrix},\qquad 
U(b,c)=\begin{pmatrix}1 & 0 \\ c^{-1} & b^\bullet (c^\bullet)^{-1}\end{pmatrix}$$
\end{defn}

We have the following:

\begin{lemma}\label{crucial}
If the variables $a,b,c,a^\bullet,b^\bullet,c^\bullet$ are such that 
$$ a b^{-1} = (b^\bullet)^{-1}a^\bullet\quad {\rm and}\quad  b^{-1} c=c^\bullet (b^\bullet)^{-1},  $$
then
$$ V(a,b)\,U(b,c)=U(a,x)\,V(x,c)\qquad {\rm iff}\qquad \left\{ \begin{matrix} x= (b^\bullet)^{-1}+a b^{-1}c \\ 
b=(x^\bullet)^{-1}+ c x^{-1}a \end{matrix}\right. $$
and moreover we have:
$$  a^{-1}\,x = x^\bullet\, (a^\bullet)^{-1} \quad {\rm and} \quad   x \,c^{-1}=(c^\bullet)^{-1}\,x^\bullet$$
\end{lemma}
\begin{proof}
We compute:
\begin{eqnarray*} V(a,b)\,U(b,c)&=&\begin{pmatrix} a b^{-1}+(b^\bullet)^{-1}c^{-1} &  
(c^\bullet)^{-1} \\ c^{-1} & b^\bullet (c^\bullet)^{-1}\end{pmatrix} \\
U(a,x)\,V(x,c)&=&\begin{pmatrix} x c^{-1} &  
(c^\bullet)^{-1} \\ c^{-1} & a^\bullet (x^\bullet)^{-1}+x^{-1}(c^\bullet)^{-1}  \end{pmatrix} 
\end{eqnarray*}
The identity between these is equivalent to the system:
$$x= (b^\bullet)^{-1}+a b^{-1}c \quad {\rm and} \quad  b^\bullet=x^{-1}+ a^\bullet (x^\bullet)^{-1}c^\bullet $$
and we get the first statement of the lemma by taking the $\bullet$ of the second equation. For the second statement, note that:
$$ a^{-1}x= a^{-1}(b^\bullet)^{-1}+b^{-1}c=b^{-1}(a^\bullet)^{-1}+c^\bullet (b^\bullet)^{-1}=x^\bullet (a^\bullet)^{-1} $$
where we have used $a^{-1}(b^\bullet)^{-1}=b^{-1}(a^\bullet)^{-1}$ and $b^{-1}c=c^\bullet (b^\bullet)^{-1}$. Analogously,
$$ x c^{-1}=(b^\bullet)^{-1}c^{-1}+a b^{-1}=(c^\bullet)^{-1} b^{-1}+(b^\bullet)^{-1}a^\bullet=(c^\bullet)^{-1}x^\bullet$$
and the lemma follows.
\end{proof}

The equations of the Lemma above may be conveniently rewritten in a non-commutative polynomial form.
We have:
\begin{lemma}
Given the non-commuting variables $a,b,c,a^\bullet,b^\bullet,c^\bullet$ subject to the relations
$a b^{-1} = (b^\bullet)^{-1}a^\bullet$ and $b^{-1} c=c^\bullet (b^\bullet)^{-1}$, the system:
\begin{equation}\label{nctsystwo} \left\{ \begin{matrix} x= (b^\bullet)^{-1}+a b^{-1}c \\ 
b=(x^\bullet)^{-1}+ c x^{-1}a \end{matrix}\right.\end{equation}
is equivalent to the single equation:
\begin{equation}\label{singlenctsys}  x \, b^\bullet= 1+a \, c^\bullet \end{equation}
The latter also implies the relations 
\begin{equation}\label{ncomm}  a^{-1}\,x = x^\bullet\, (a^\bullet)^{-1} \quad {\rm and} \quad   x \,c^{-1}=(c^\bullet)^{-1}\,x^\bullet
\end{equation}
\end{lemma}
\begin{proof}
Multiplying the first equation of the system \eqref{nctsystwo} on the right by $b^\bullet$, and using $b^{-1} c=c^\bullet (b^\bullet)^{-1}$
gives eq.\eqref{singlenctsys}. Conversely, the latter implies that $x=(b^\bullet)^{-1}+a b^{-1}c$ and the relations \eqref{ncomm} follow. 
Now multiplying the latter on the left
by $b^\bullet$ yields $b^\bullet\, x=1+a^\bullet\, c$ by use of $a b^{-1} = (b^\bullet)^{-1}a^\bullet$. 
This gives $x^\bullet\, b=1+c^\bullet \, a$, upon applying $\bullet$. 
Finally, multiplying this on the left by $(x^\bullet)^{-1} $  yields the second equation of the system \eqref{nctsystwo}, by use of 
$ (x^\bullet)^{-1} c^\bullet =c x^{-1}$, and the lemma follows.
\end{proof}

\subsection{Solution and NC positive Laurent property}

Let us consider the NC $A_1$ $T$-system with initial data $(\bm,x_\bm,x_\bm^\bullet)$. 
As explained above, without loss of generality, we may restrict ourselves to
finding an expression for $T_{j,k}$ with $k\geq m_j$, namely above the boundary path.

\begin{defn}
Let $j_0,j_1\in \Z$ to be the lower and upper projections of $(j,k)$ onto the path $\bm$ 
defined as follows. $j_0$ is the largest integer $\ell$ such that $k-j=m_\ell-\ell$. $j_1$ is the smallest integer $\ell$
such that $k+j=\ell+m_\ell$. The section of $\bm$ between $j_0$ and $j_1$ is also called the projection of $(j,k)$
onto $\bm$.
\end{defn}

\begin{defn}
For $x,y\in \Z$, $x\leq y$, we define the $2\times 2$ matrix $M_\bm(x,y)$ as follows. We consider the section of $\bm$ between $x$ and $y$
namely $\{m_j\}_{j\in [x,y]}$. From left to right it is a succession of up (if $m_{j+1}=m_j+1$) and down (if $m_{j+1}=m_j-1$) steps.
Let $M_\bm(x,x)={\mathbb I}$ the $2\times 2$ identity matrix. We define $M_\bm(x,j)$, $j\geq x$ by induction as:
$$M_\bm(x,j+1)=M_\bm(x,j)\times\left\{ 
\begin{matrix} U(t_j,t_{j+1}) & {\rm if} \, m_{j+1}=m_j+1 \\ V(t_j,t_{j+1}) & {\rm if} \, m_{j+1}=m_j-1 \end{matrix} \right. $$
In other words $M_\bm(x,y)$ is the product of $U,V$ matrices along the section of $\bm$ between $x$ and $y$, 
with $U$ for up steps, and $V$ for down steps.
\end{defn}

The strategy for solving the NC $A_1$ $T$-system is as follows. We start from the situation when the point $(j,k)$ belongs to 
the initial data path $\bm$, in which case the solution $T_{j,k}$ is trivially $t_j$, the initial data assignment. Next we ``mutate"
the initial data by local moves consisting of replacing a succession of down-up steps with a succession of up-down steps. 
This in turn corresponds to one backward application of the matrix identity of Lemma \ref{crucial} in which the old variables $x,x^\bullet$
must then be replaced by $b,b^\bullet$ on the new (backward mutated) path. This however conserves the value of the matrix product,
allowing us to write a general formula for the solution that is independent of the path. We have the following:

\begin{thm}\label{solution}
The solution $T_{j,k}$ of the NC $A_1$ $T$-system with initial data $(\bm,x_\bm,x_\bm^\bullet)$
is given by
\begin{equation}\label{solt}
T_{j,k}=\left( M_\bm(j_0,j_1) \right)_{1,1} \, t_{j_1}
\end{equation}
\end{thm}
\begin{proof}
By induction on the shape of $\bm$. It is clear from the definitions that the projection of $(j,k)$ onto $\bm$ should start with a down step
and end with an up step. Let us first consider the case of a ``maximal" path section $\bp$ between $j_0$ and $j_1$ made of a succession
of up steps followed by a succession of down steps, and an associated assignment of initial values $\{\theta_i,\theta_i^\bullet\}_{i\in [j_0,j_1]}$
where $\theta_i=T_{i,p_i}$ is the solution of the system at the point $(i,p_i)$. 
The corresponding matrix $M_\bp(j_0,j_1)$ is a (lower triangular) product $\mathcal U$
of $U$ matrices 
followed by a (upper triangular) product $\mathcal V$ of $V$ matrices. As a consequence, $\left(M_\bp(j_0,j_1)\right)_{1,1}={\mathcal U}_{1,1}
{\mathcal V}_{1,1}$. We easily compute ${\mathcal U}_{1,1}=1$ as all $(1,1)$ elements of $U$ matrices are equal to $1$.
We also compute:
$${\mathcal V}_{1,1} =\prod_{i=j}^{j_1-1} V(\theta_i,\theta_{i+1})_{1,1}=\theta_j \theta_{j_1}^{-1}$$
as the product is telescopic. Noting that $\theta_{j_1}=t_{j_1}$, and $\theta_j=T_{j,k}$, this finally gives 
$$ \left( M_\bp(j_0,j_1) \right)_{1,1} \, t_{j_1}=\theta_j=T_{j,k} $$
so the formula \eqref{solt} holds for $\bm$ replaced by the maximal section of path $\bp$. We may now ``peel" $\bp$ by successive
mutations so as to reach $\bm$ in finitely many steps. Each such step involves substituting a product of the form $UV$  with a product $V'U'$
by backward use of the identity of Lemma \ref{crucial}. As each such substitution preserves the value of the matrix $M(j_0,j_1)$,
we finally get $M_\bm(j_0,j_1)=M_\bp(j_0,j_1)$, and the theorem follows.
\end{proof}

\begin{remark}
Note that Theorem \ref{solution} immediately implies Theorem \ref{main}, as the non-zero entries of the $U,V$ matrices are all
Laurent monomials of the variables $t_j$ and $t_j^\bullet$, with coefficient $1$, and the entries of the matrix
$M_\bm(j_0,j_1)$ are therefore Laurent polynomials of $\{t_j,t_j^\bullet\}$ with non-negative integer coefficients.
%These coefficients are interpreted as counting paths on an oriented graph, or dimer coverings of a ladder graph in the next sections.
\end{remark}

\begin{remark}\label{wello}
Note that the solution given by Theorem \ref{solution} produces for $T_{j,k}$ a Laurent polynomial in which each contributing monomial
is well-ordered, namely indices of $t_i^{\pm1}$ and $(t_i^\bullet)^{\pm 1}$ are strictly increasing from left to right. This form is unique, as
applications of the relations \eqref{quasicoinit} on the initial data would change the ordering.
\end{remark}

\begin{example}\label{flatex}
Let us consider the case of the ``flat" initial data path $\bm^{(0)}$, with $m_{i}^{(0)}=i$ mod 2, and with the assignments
$T_{i,i\, {\rm mod}\, 2}=t_i$, $i\in \Z$. Then we have:
\begin{eqnarray}T_{3,3}&=& \left( V(t_1,t_2)U(t_2,t_3)V(t_3,t_4)U(t_4,t_5)\right)_{1,1}\, t_5 \nonumber \\
&=&\left( \begin{pmatrix} t_1 t_2^{-1}+(t_2^\bullet)^{-1}t_3^{-1} &  
(t_3^\bullet)^{-1} \\ t_3^{-1} & t_2^\bullet (t_3^\bullet)^{-1}\end{pmatrix}.
\begin{pmatrix} t_3 t_4^{-1}+(t_4^\bullet)^{-1}t_5^{-1} &  
(t_5^\bullet)^{-1} \\ t_5^{-1} & t_4^\bullet (t_5^\bullet)^{-1}\end{pmatrix} \right)_{1,1}\, t_5 \nonumber \\
&=&(t_1 t_2^{-1}+(t_2^\bullet)^{-1}t_3^{-1})(t_3 t_4^{-1}t_5+(t_4^\bullet)^{-1})+(t_3^\bullet)^{-1}  \nonumber \\
&=&t_1 t_2^{-1}t_3 t_4^{-1}t_5+t_1 t_2^{-1}(t_4^\bullet)^{-1}+ (t_2^\bullet)^{-1}t_4^{-1}t_5+ (t_2^\bullet)^{-1}t_3^{-1}(t_4^\bullet)^{-1}
+(t_3^\bullet)^{-1} \label{t33ex}
\end{eqnarray}
Note the ordering of the terms in each monomial, a manifestation of the well-ordered structure of Remark \ref{wello}.
\end{example}

\subsection{NC Networks}

In this section, we express the solution to \eqref{ncaonet} in terms of any admissible data
as a non-commutative network partition function, namely as the partition function of paths
with non-commutative weights on a suitably defined oriented graph.

In the classical and quantum cases, the matrices $U,V$ have been interpreted as transfer matrices (or ``chips") for
paths on weighted graphs. In the NC case, the interpretation is the same, but the paths now receive NC step weights,
to be multiplied in the same order as the corresponding steps. More precisely, to each matrix $V_{a,b}$, $U_{a,b}$ we
associate an oriented graph with two entry connectors labeled $1,2$ on the left and two exit connectors $1,2$ on the right, and
an oriented step $i\to j$ whenever the corresponding matrix element $(i,j)$ is non-zero. We attach the value of this matrix element to the corresponding oriented step on the graph. This produces two weighted oriented graphs as follows (all edges are oriented from left to right):
\begin{eqnarray}\label{udnet}
&&\epsfxsize=10cm \epsfbox{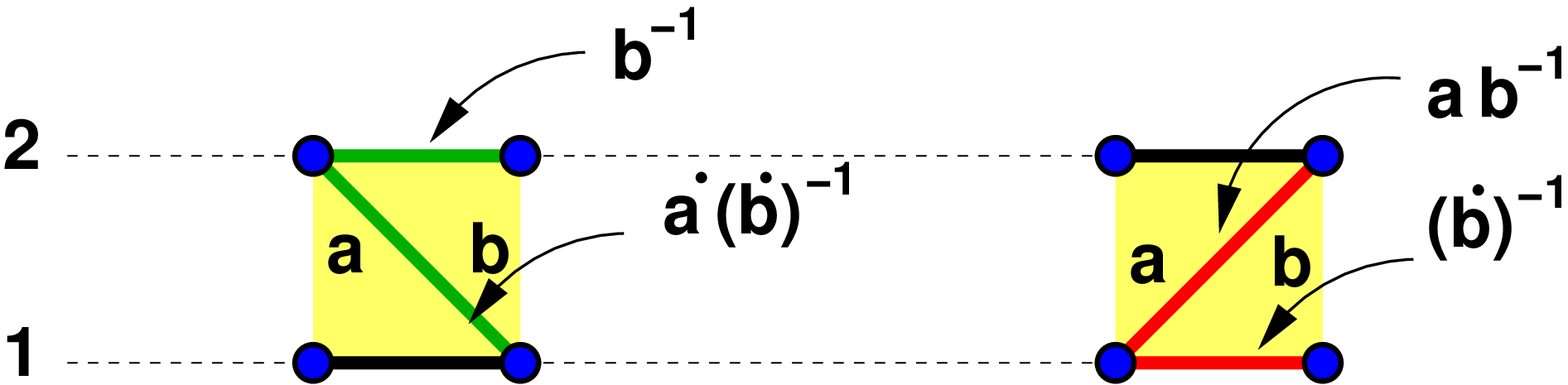} \nonumber \\
&& \qquad  \qquad  \quad   U(a,b) \qquad  \qquad  \qquad \qquad \quad \ 
V(a,b)
\end{eqnarray}
where we have indicated the non-trivial step weights as NC Laurent monomials of $a,b,a^\bullet,b^\bullet$.
Note that the variables $a,b$ become face labels in the pictorial representation.

A NC network is the graph obtained by the concatenation of such chips,
forming a chain where the exit connectors $1,2$ of each chip in the
chain are identified with the entry connectors of the next chip in the
chain, while face labels are well-defined. The latter condition
imposes that $U$ and $V$ parameters themselves form a chain
$a_1,a_2,...$, for instance the matrix:
\begin{equation}\label{duex}
W=U(a_1,a_2)U(a_2,a_3)V(a_3,a_4)U(a_4,a_5)V(a_5,a_6)
\end{equation}
corresponds to the network:
$$
%\raise -.5truecm 
\epsfxsize=7cm \epsfbox{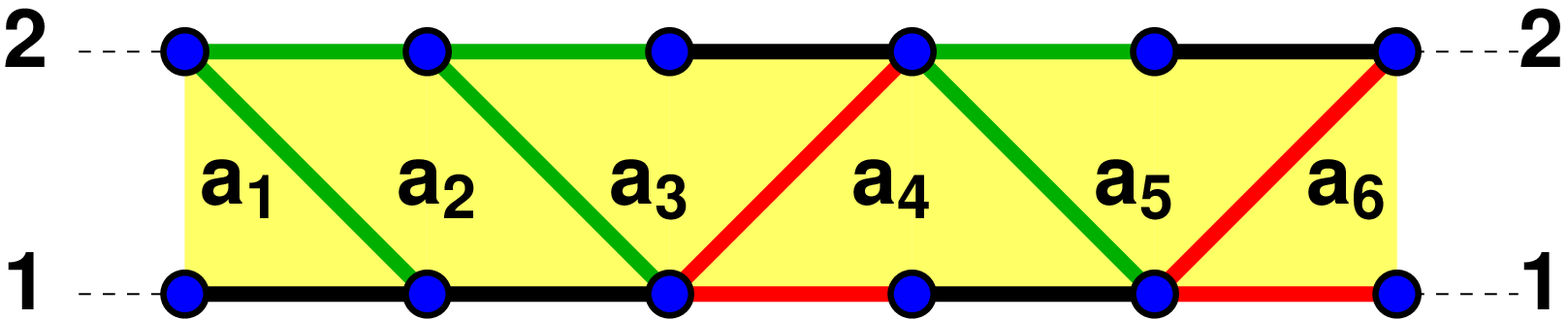}
$$
The partition function of a NC network with weighted adjacency matrix $W$,
with entry connector $i$ and exit connector $j$ is the matrix element $W_{i,j}$.  It
is the sum over paths from entry $i$ to exit $j$ of the product of
edges weights, taken in the order they are traversed.

We are now in position to clarify the interpretation of the non-negative integer coefficients in the positive NC 
Laurent property of the solutions of the NC $A_1$ $T$-system, expressed in terms of some arbitrary initial data.
To each matrix $M_\bm(j_0,j_1)$ we may associate a NC network
${\mathcal N}_\bm(j_0,j_1)$ made of the $j_1-j_0$
concatenated  ($U$/$V$)  chips corresponding to the (up/down) step succession in the relevant section of $\bm$. 
In particular, the entry $(1,1)$ of this matrix is interpreted as the partition function for paths of $j_1-j_0$
steps from entry point $1$ to exit point $1$ on the network graph ${\mathcal N}_\bm(j_0,j_1)$, namely the sum over all paths
weighted by the product of their step weights in the order in which they are taken. This is summarized in the following:

\begin{thm}\label{netsolthm}
The solution $T_{j,k}$ of the NC $A_1$ $T$-system with initial data $(\bm,x_\bm,x_\bm^\bullet)$ is the partition function
for paths from entry connector $1$ to exit connector $1$ on the NC network ${\mathcal N}_\bm(j_0,j_1)$ associated to the initial data,
multiplied by $t_{j_1}$.
\end{thm}

\begin{example}\label{flapathex}
Let us consider again the case of the ``flat" initial data path $\bm^{(0)}$ of Example \ref{flatex}. The expression \eqref{t33ex}
for $T_{3,3}$ is the sum of five monomials. These are interpreted as the five paths of 4 steps, from connector 1 to connector 1,
on the network ${\mathcal N}_{\bm^{(0)}}(1,5)$:
$$ {\mathcal N}_{\bm^{(0)}}(1,5)= \raise -1.5cm  \hbox{\epsfxsize=5cm \epsfbox{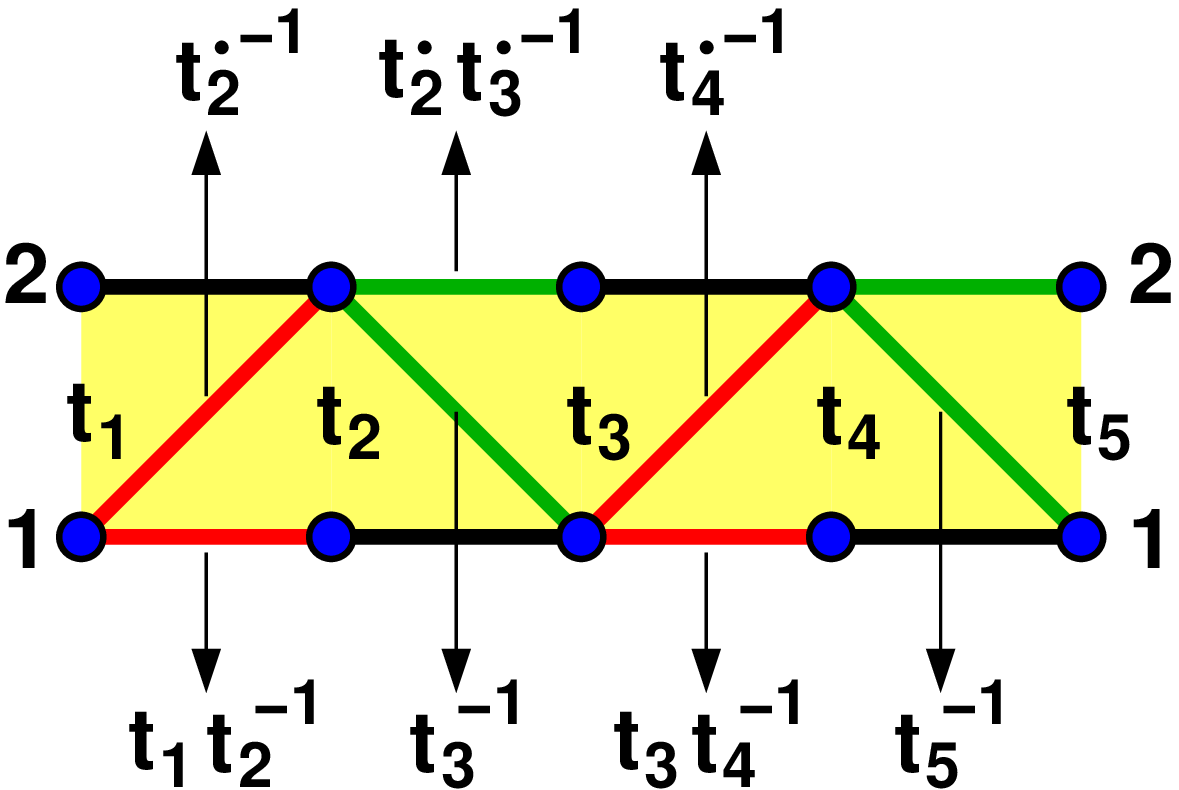} }$$
where we have indicated the non-trivial step weights. The five paths from entry 1 on the left to exit 1 on the right are:
$$  \begin{matrix} {\rm path:} 
&\hbox{\epsfxsize=2.cm \epsfbox{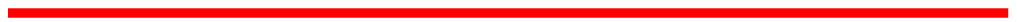} }
& \hbox{\epsfxsize=2.cm \epsfbox{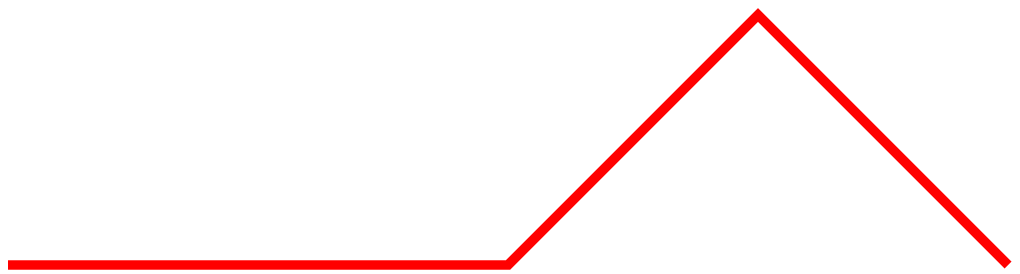} } 
&\hbox{\epsfxsize=2.cm \epsfbox{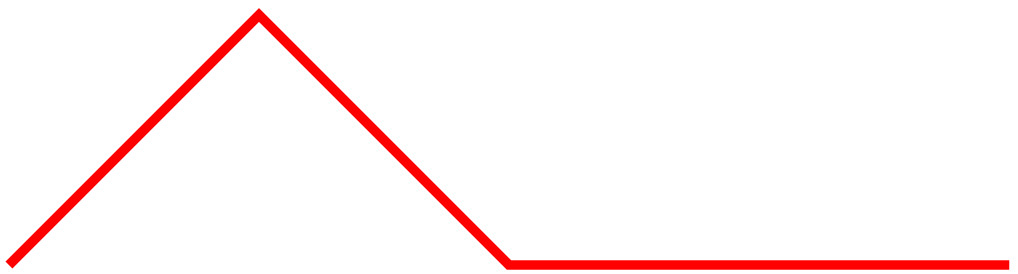} }
&\hbox{\epsfxsize=2.cm \epsfbox{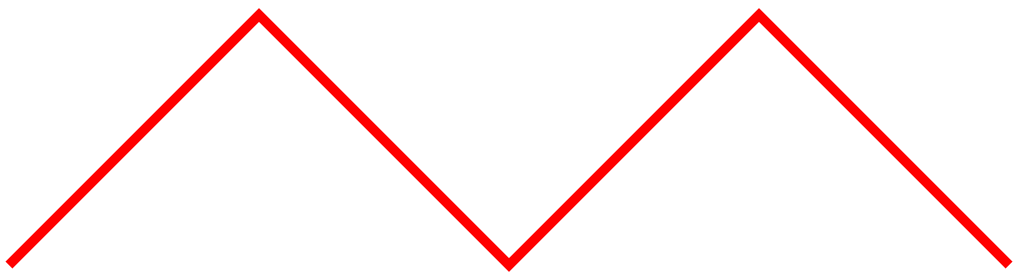} }  
& \hbox{\epsfxsize=2.cm \epsfbox{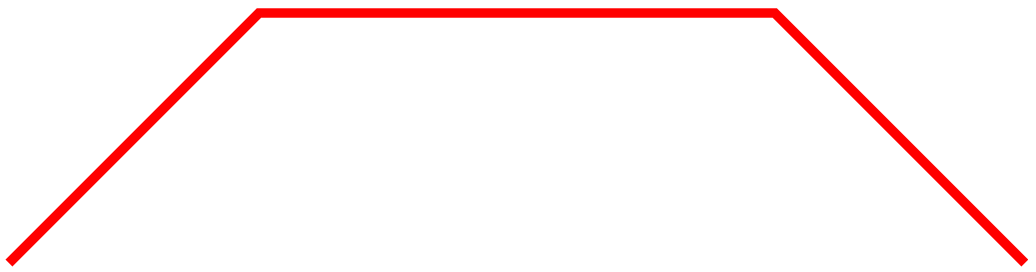} } 
\\
& & & & & \\
{\rm weight:} & t_1t_2^{-1}t_3t_4^{-1}t_5 &t_1t_2^{-1} (t_4^\bullet)^{-1} &(t_2^\bullet)^{-1}t_4^{-1}t_5 
&(t_2^\bullet)^{-1}t_3^{-1} (t_4^\bullet)^{-1}&(t_3^\bullet)^{-1} \end{matrix} $$
where we have also represented their monomial contributions to $ T_{3,3}=\left({\mathcal M}_{\bm^{(0)}}(1,5)\right)_{1,1}t_5$ of \eqref{t33ex}.
\end{example}

%\begin{example}
%\end{example}

\subsection{NC dimers}

Another interpretation of the solution $T_{j,k}$  holds in terms of a dimer partition function on a suitable 4-6 ladder graph,
entirely determined by the initial data.
A dimer model on a given bipartite graph is a weighted statistical ensemble of configurations of ``dimers" occupying the edges of the graph in such a way that each vertex of the graph is covered by exactly one dimer. The weight of the configuration is usually the product of weights of local configurations of dimers say around each face in the case of a graph embedded in a surface.

In this section, we extend the definition to dimer models with non-commutative weights on particular ladder-like graphs.

\subsubsection{The case of flat initial data path $\bm^{(0)}$.}
In this case, the corresponding network of Theorem \ref{netsolthm} is the concatenation of a succession of $U$/$V$ type chips, between the positions $j_0$ and $j_1$. In this section we show that the corresponding partition function of paths from entry connector $1$ to exit connector $1$
may be recast into the partition function of dimers on a suitable ``ladder" graph.

\begin{defn} The ladder graph ${\mathcal L}_N$ of length $N$ is a vertex-bicolored (black and white) 
planar graph with $2N-2$ vertices say at integer points in the rectangle 
$\{1,2\}\times \{1,2,...,N-1\}\subset \Z^2$, $N-2$ inner square faces labeled $2,3,...,N-1$, 
and two boundary ``faces" labeled 1,2 adjacent respectively
to the leftmost and rightmost vertical edge. By convention we color in black the vertex at the lower left corner.
Each face of the ladder graph is equipped with a pair $t_j,t_j^\bullet$, $j=1,2,...,N$
of non-commutative weights, where $j$  runs over the face labels in ${\mathcal L}_N$.
\end{defn}

\begin{defn}
The Non-Commutative (NC) dimer model on the ladder graph ${\mathcal L}_N$ is defined by attaching to each dimer configuration on 
${\mathcal L}_N$ a non-commutative weight as follows. The total weight of a given configuration is the product (in this order) of the
weights of faces $1,2,...,N$.
The left and right boundary face receive respectively the weight $t_i^{1-D_i}$, $i\in \{1,N\}$, where $D_i$ 
is the number of dimers ($\in \{0,1\}$) adjacent to the corresponding boundary face. Each inner face labeled $i\in \{2,3,...,N-1\}$ receives a weight 
as indicated in the table below, depending on the dimer configuration around the face, and on the parity of the face label:
\begin{equation}\label{rulesq}  \begin{matrix} \begin{matrix} {\rm dimer} \\ {\rm configuration} \end{matrix}
&\raise -.3cm\hbox{\epsfxsize=1.cm \epsfbox{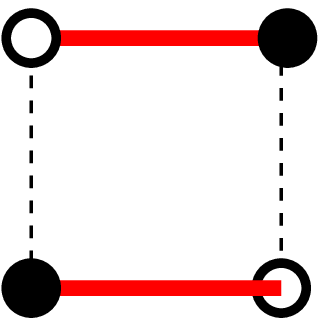} }\ \ 
& \raise -.3cm\hbox{\epsfxsize=1.cm \epsfbox{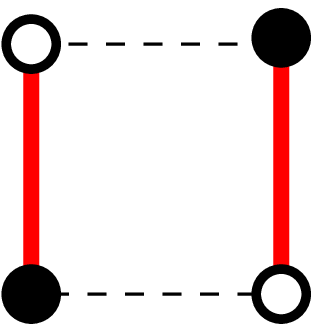} } \ \ 
&\raise -.3cm\hbox{\epsfxsize=1.cm \epsfbox{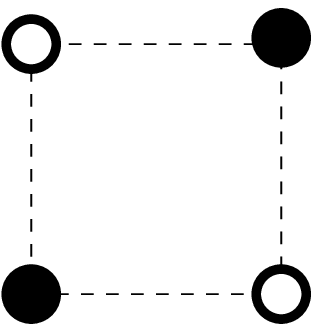} }\ \ 
&\raise -.3cm\hbox{\epsfxsize=1.cm \epsfbox{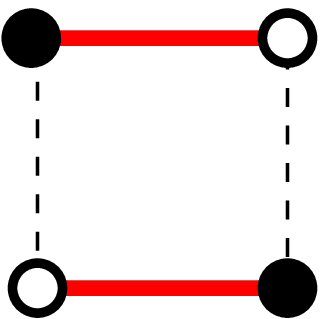} }\ \ 
& \raise -.3cm\hbox{\epsfxsize=1.cm \epsfbox{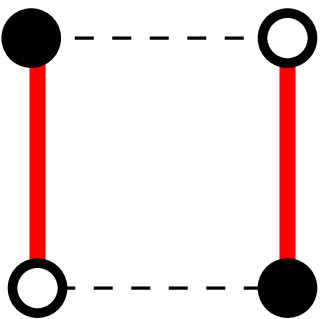} } \ \ 
&\raise -.3cm\hbox{\epsfxsize=1.cm \epsfbox{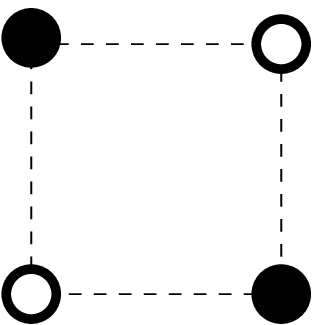} }\ \ 
& {\rm other} 
\\
& & & & & & \\
\begin{matrix} {\rm face} \\ {\rm weight} \end{matrix}
&t_{i}^{-1}
&(t_{i}^\bullet)^{-1}
&t_{i}^\bullet 
&(t_{i}^\bullet)^{-1}
&t_{i}^{-1}
&t_{i}
& 1
\\
\end{matrix} 
\end{equation}
(the face label in the first three cases is even, and odd in the next three). The partition function $Z({\mathcal L}_N)$ of the NC dimer model 
on the NC weighted ladder graph ${\mathcal L}_N$ is the sum over all dimer configurations on ${\mathcal L}_N$ of the
corresponding weights.
\end{defn}

\begin{thm}\label{flatdim}
The solution $T_{j,k}$ of the NC $A_1$ $T$-system with fundamental initial data 
$(\bm^{(0)},x_{\bm^{(0)}},x_{\bm^{(0)}}^\bullet)$ is the partition function $Z({\mathcal L}_N)$
for NC dimers on the ladder graph ${\mathcal L}_N$ with $N=2k+1$ and with face variables $t_{j_0},t_{j_0}^\bullet,t_{j_0+1},t_{j_0+1}^\bullet,\cdots t_{j_1},t_{j_1}^\bullet$, where $j=\frac{j_0+j_1}{2}$ and $k=\frac{j_1-j_0}{2}$.
\end{thm}
\begin{proof}
By bijection. We use the known bijection between paths of length $2k$ from entry connector 1 to exit connector 1 
on the network ${\mathcal N}_{\bm^{(0)}}(j_0,j_1)$ and the dimer configurations on ${\mathcal L}_{2k+1}$.
On even faces, the bijection maps the five possible local path configurations to five local dimer configurations according to the following dictionary:
\begin{equation}\label{fivecases} \raisebox{-2.cm}{\hbox{\epsfxsize=14.cm \epsfbox{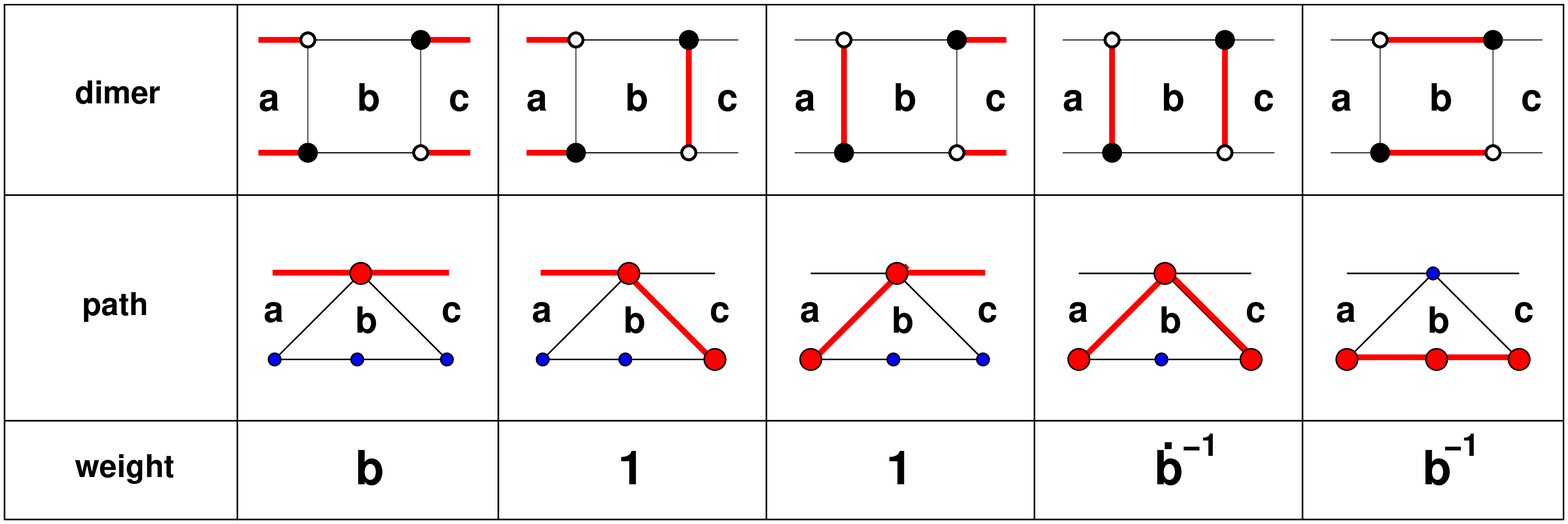}}}
\end{equation}
where we have indicated the $b,b^\bullet$ dependence of the step weights, which we take to be the weight of the corresponding
square face labeled $b$ of the dimer model.
Odd faces are treated similarly,  with $b$ and $b^\bullet$ interchanged. The theorem follows by inspection of the weights,
namely by collecting all contributions of path weights that involve the variables $b,b^\bullet$ of any given face.
\end{proof}

\begin{example}
Let us present the dimer version of Examples \ref{flatex} and \ref{flapathex}. 
The five paths contributing to $T_{3,3}$ correspond bijectively to the following dimer configurations:
$$ \hbox{\epsfxsize=14.cm \epsfbox{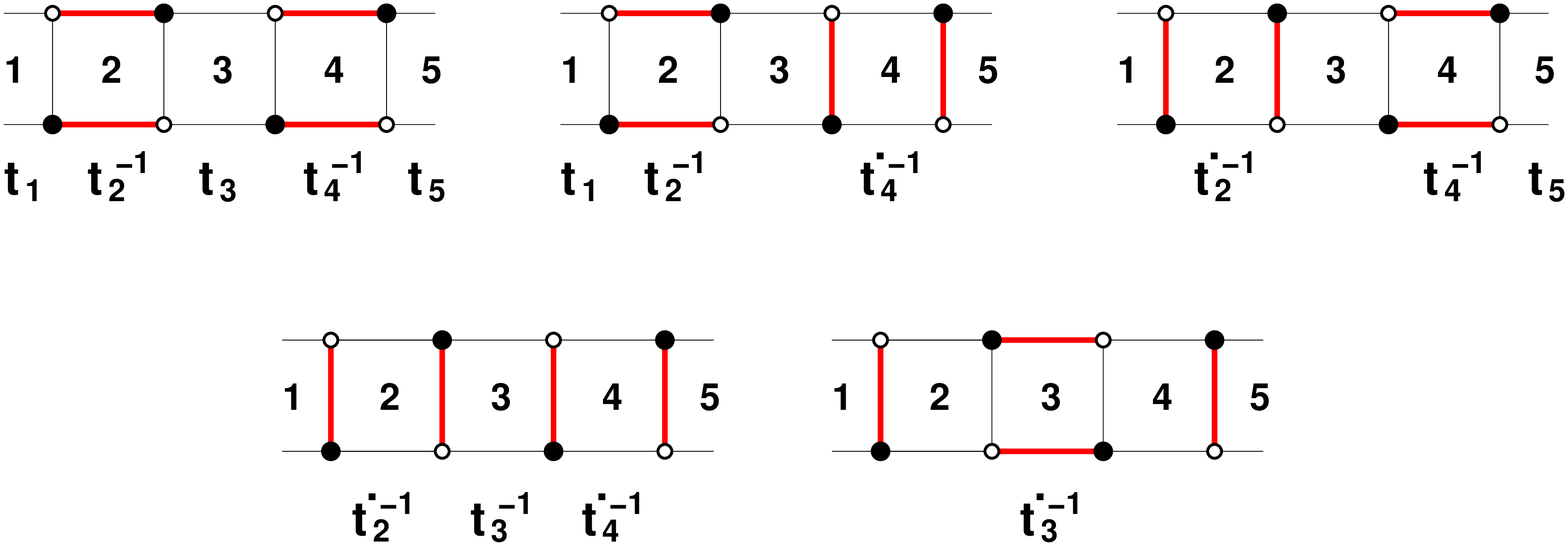}}$$
where we have indicated all the non-trivial face weights.
\end{example}

\subsubsection{General initial data path $\bm$}

In \cite{DF13}, it was shown that for general initial data the partition function of network paths giving rise to $T_{j,k}$
may be reformulated as the partition function of dimers on some $4-6$ generalized ladder graph ${\mathcal L}_\bm(j_0,j_1)$, 
made of chains of a succession of labeled square and hexagonal faces (with labels $j_0+1,j_0+2,...,j_1-1$)
with an additional left boundary face labeled $j_0$ and a right one labeled $j_1$.
More precisely, successions of up-down or down-up steps of $\bm$ give rise to square faces (as in the flat case of previous section):
\begin{eqnarray*}
V(a,b)U(b,c) \quad \to\quad
\raisebox{-.4cm}{\hbox{\epsfxsize=2.3cm \epsfbox{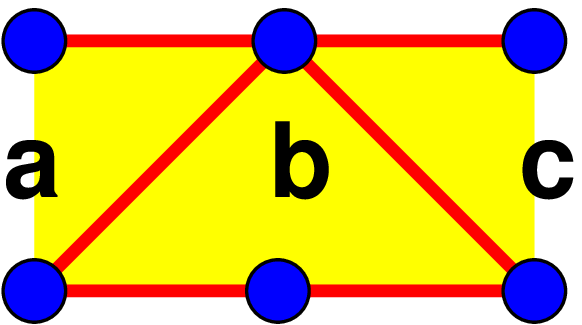}}} 
\quad &\to&\quad   \raisebox{-.5cm}{\hbox{\epsfxsize=2.2cm \epsfbox{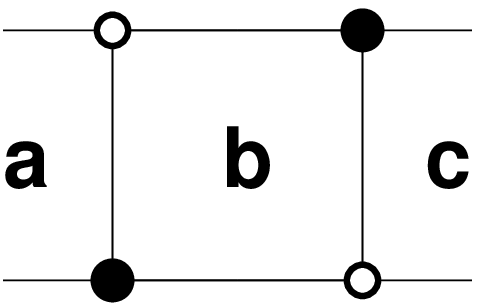}}}\\
U(a,b)V(b,c) \quad \to\quad
\raisebox{-.4cm}{\hbox{\epsfxsize=2.3cm \epsfbox{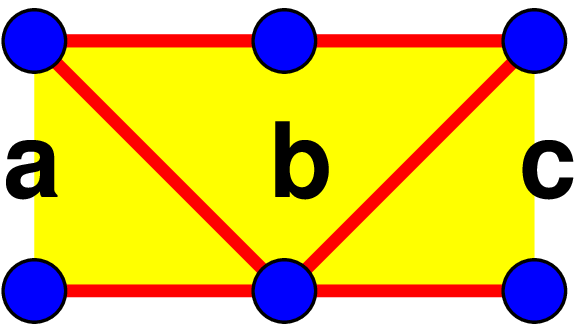}}} 
 \quad &\to&\quad  \raisebox{-.5cm}{\hbox{\epsfxsize=2.2cm \epsfbox{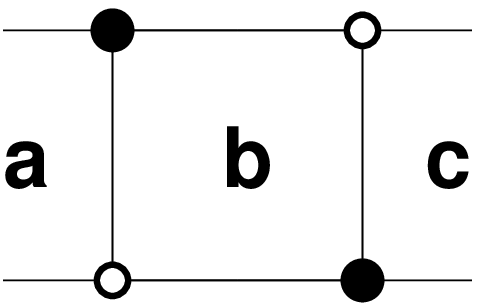}}}
\end{eqnarray*}
while successions of up-up or down-down steps of $\bm$ give rise to hexagons as follows:
\begin{eqnarray*}
U(a,b)U(b,c) \quad \to\quad
\raisebox{-.6cm}{\hbox{\epsfxsize=2.4cm \epsfbox{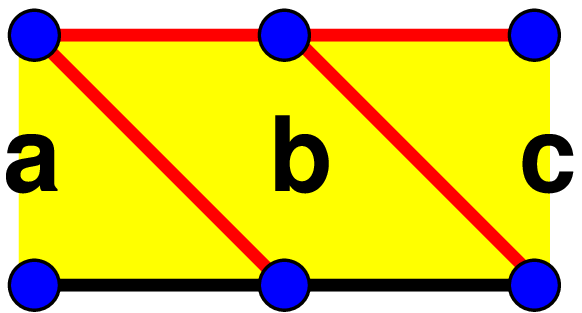}}} 
\quad &\to&\quad \raisebox{-.5cm}{\hbox{\epsfxsize=2.9cm \epsfbox{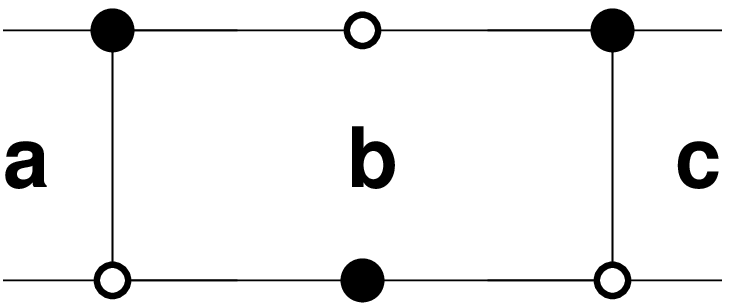}}} \\
V(a,b)V(b,c) \quad \to\quad
\raisebox{-.6cm}{\hbox{\epsfxsize=2.4cm \epsfbox{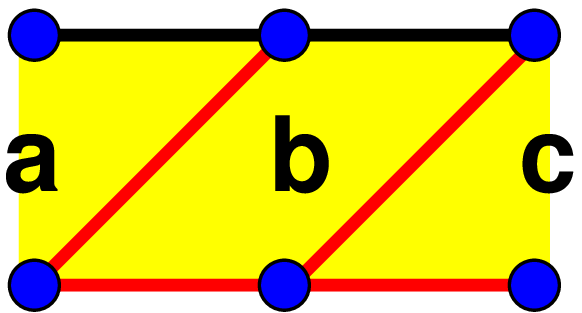}}} \quad &\to&\quad  
\raisebox{-.5cm}{\hbox{\epsfxsize=2.9cm \epsfbox{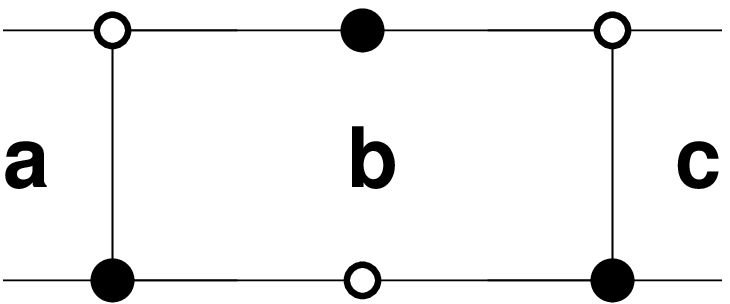}}}\end{eqnarray*}
One way of understanding this is via the generalized bijection between
arbitrary paths from 1 to 1 on the network ${\mathcal N}_\bm(j_0,j_1)$ and dimer configurations on 
the generalized ladder ${\mathcal L}_\bm$. 

As before the bijection is local, and in addition to the five cases \eqref{fivecases} giving rise to squares, we have the following
four cases giving rise to hexagons on even faces (with a black lower left vertex):
\begin{equation}\label{rulhex}  \hbox{\epsfxsize=16.cm \epsfbox{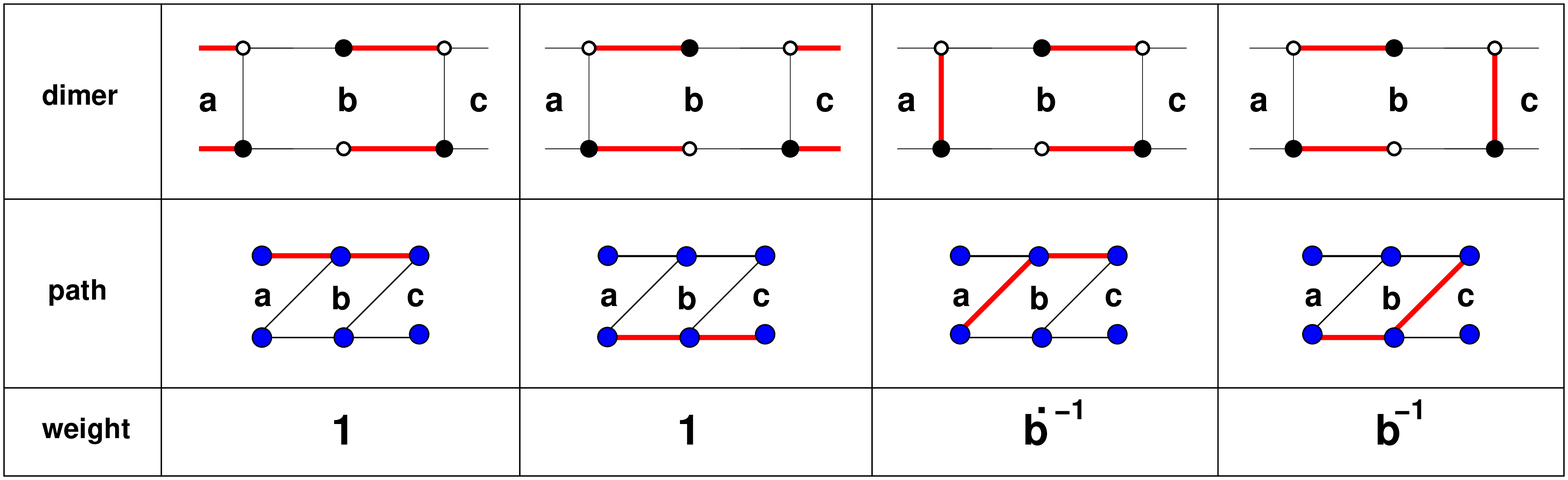}}\end{equation}
and similarly for the odd one (those with white lower left vertex), with $b$ and $b^\bullet$ interchanged.

\begin{defn}
We define the NC dimer model on a generalized ladder graph ${\mathcal L}_\bm(j_0,j_1)$ as the statistical 
ensemble of dimer configurations on ${\mathcal L}_\bm(j_0,j_1)$, each receiving the product of face weights from
$j_0$ to $j_1$ (from left to right), each face being weighted according to the previous rules of 
\eqref{rulesq} and their odd counterparts with $b\leftrightarrow b^\bullet$ for square faces and boundary faces and to those of \eqref{rulhex}
and their odd counterparts with $b\leftrightarrow b^\bullet$
for hexagonal ones.
The partition function $Z({\mathcal L}_\bm(j_0,j_1))$ of the NC dimer model 
on the NC weighted generalized ladder graph ${\mathcal L}_\bm(j_0,j_1)$ is the sum over all dimer configurations on 
${\mathcal L}_\bm(j_0,j_1)$ of the
corresponding weights.
\end{defn}

\begin{thm}
The solution $T_{j,k}$ of the NC $A_1$ $T$-system with arbitrary initial data
$(\bm,x_{\bm},x_{\bm}^\bullet)$ is the partition function $Z({\mathcal L}_\bm(j_0,j_1))$
for NC dimers on the generalized ladder graph ${\mathcal L}_\bm(j_0,j_1)$ with $N=2k-1$ faces with variables 
$t_{j_0},t_{j_0}^\bullet,t_{j_0+1},t_{j_0+1}^\bullet,\cdots t_{j_1},t_{j_1}^\bullet$, where $j=\frac{j_0+j_1}{2}$ and $k=1+\frac{j_1-j_0}{2}$.
\end{thm}
\begin{proof}
By bijection. We use the above-mentioned bijection between the paths from 1 to 1 on ${\mathcal N}_\bm(j_0,j_1)$ and the dimer
configurations on ${\mathcal L}_\bm(j_0,j_1)$, and collect the weights involving the variables $t,t^\bullet$ of any given face to
recover the rules \eqref{rulesq} and \eqref{rulhex}.
\end{proof}

\begin{example}
Let us express $T_{2,4}$ in terms of initial data $T_{0,2}=a$, $T_{1,1}=b$, $T_{2,0}=c$, $T_{3,1}=d$, $T_{4,0}=e$, and $T_{5,1}=f$.
We have the following situation:
$$ \hbox{\epsfxsize=6.cm \epsfbox{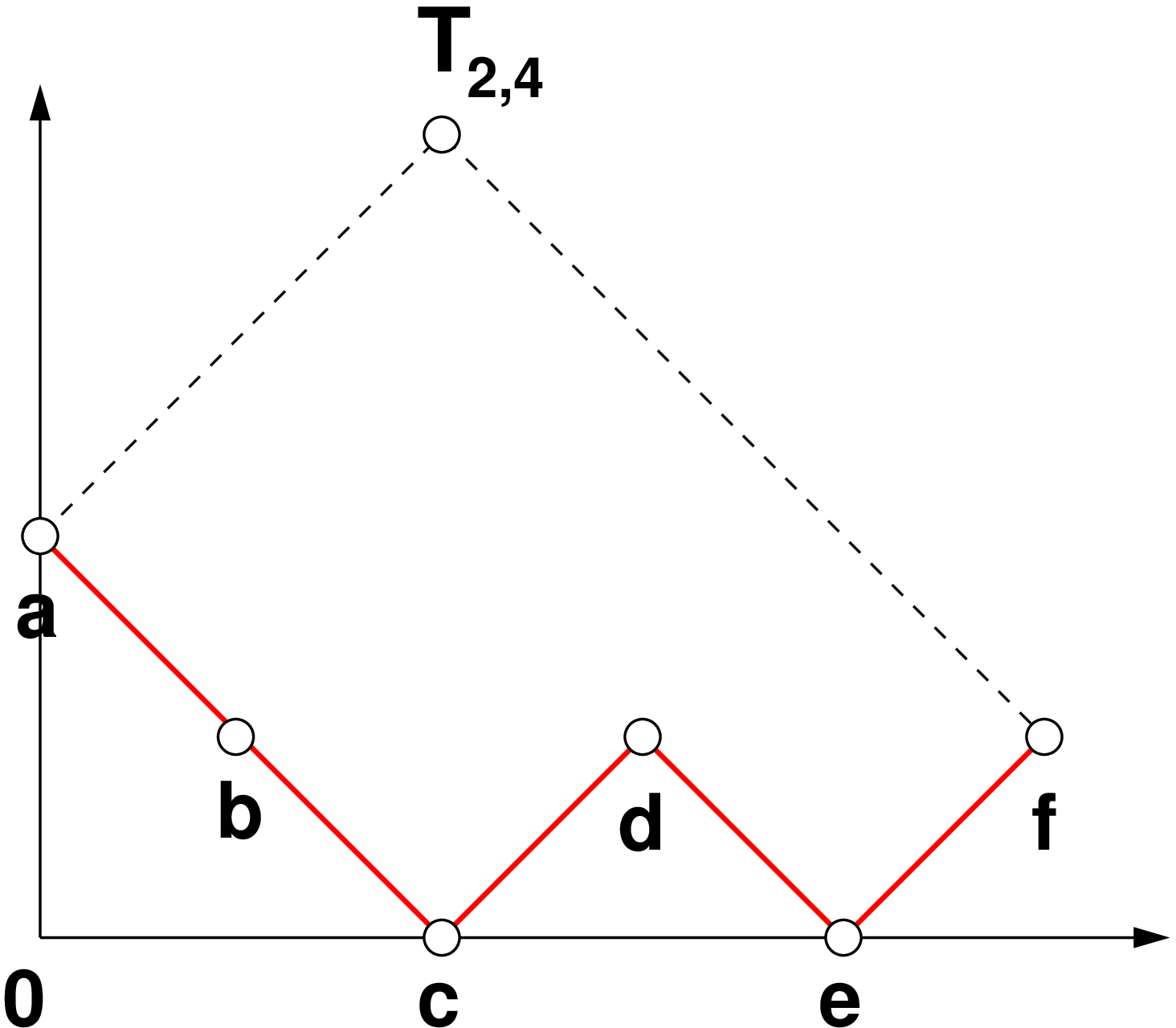}}   $$
with $j_0=0$ and $j_1=5$.
The solution from Theroem \ref{solution} is 
\begin{eqnarray*}T_{2,4}&=&\left( V(a,b)V(b,c)U(c,d)V(d,e)U(e,f)\right)_{1,1} f\\
&=&(b^\bullet)^{-1} d^{-1}e^{-1}+(b^\bullet)^{-1} e^{-1}f+(b^\bullet)^{-1} c^\bullet (d^\bullet)^{-1}
+a(b^\bullet)^{-1}(c^\bullet)^{-1}d^{-1}(e^\bullet)^{-1}\\
&&\quad + a b^{-1} (c^\bullet)^{-1}e^{-1}f+ab^{-1}(d^\bullet)^{-1}+a c^{-1}(e^\bullet)^{-1}+ac^{-1}d e^{-1}f
\end{eqnarray*}
To the relevant portion of initial data path between $j_0=0$ and $j_1=5$, we associate the following generalized  ladder graph:
$$  {\mathcal L}_\bm(j_0,j_1)=\ \  \raise-.5cm \hbox{\epsfxsize=6.cm \epsfbox{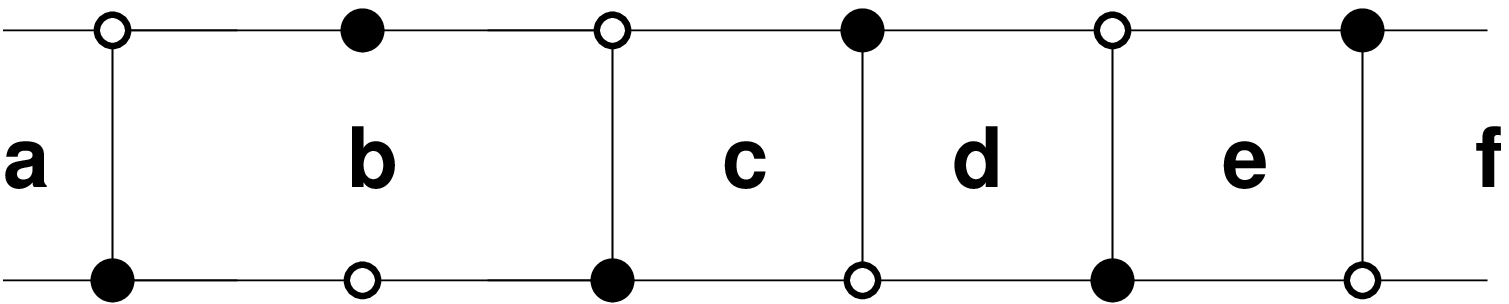}}  $$
and the eight dimer configurations contributing to $T_{2,4}$ are:
$$
\begin{matrix}
\raisebox{.0cm}{\hbox{\epsfxsize=4.8cm \epsfbox{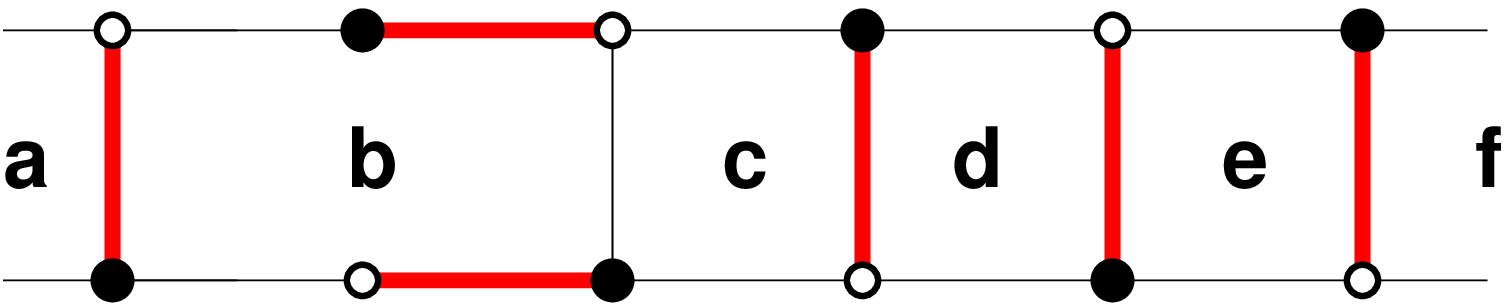}}} & \raisebox{.0cm}{\hbox{\epsfxsize=4.8cm \epsfbox{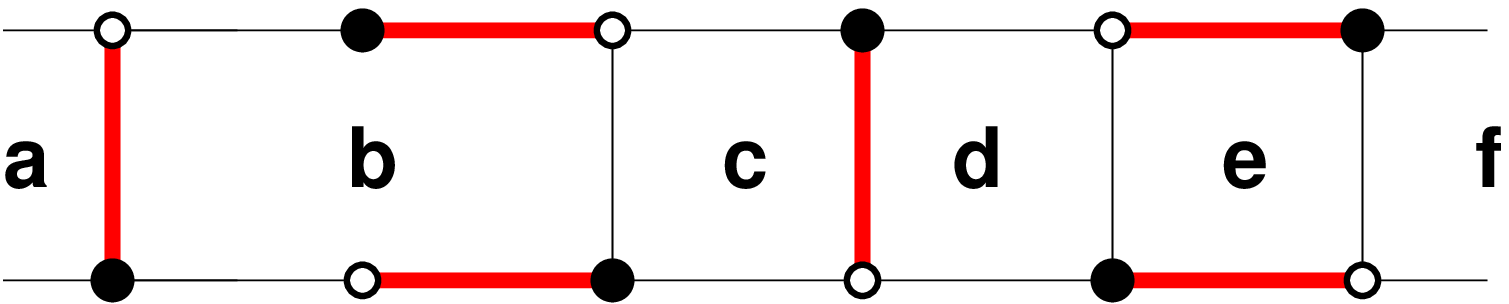}}} &
\raisebox{.0cm}{\hbox{\epsfxsize=4.8cm \epsfbox{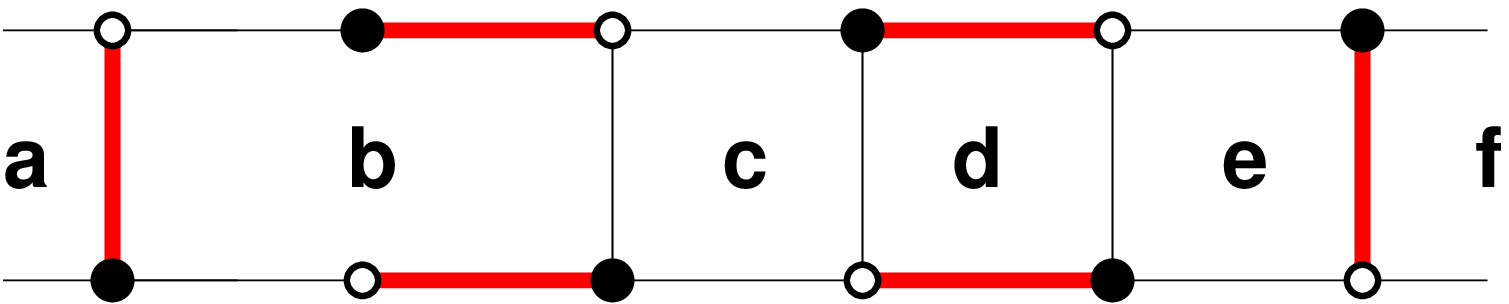}}} \\
(b^\bullet)^{-1} d^{-1}e^{-1} & (b^\bullet)^{-1} e^{-1}f & (b^\bullet)^{-1} c^\bullet (d^\bullet)^{-1}\\
 & & \\
\raisebox{-.0cm}{\hbox{\epsfxsize=4.8cm \epsfbox{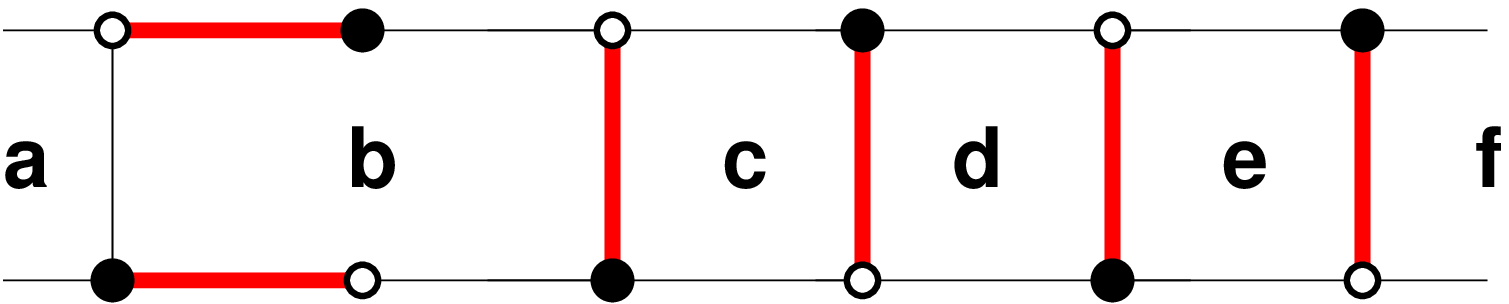}}} & \raisebox{-.0cm}{\hbox{\epsfxsize=4.8cm \epsfbox{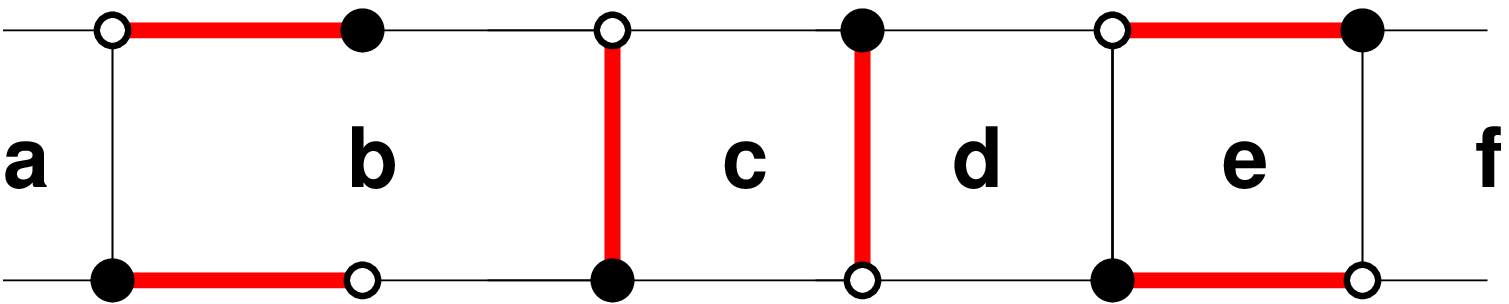}}} &
\raisebox{-.0cm}{\hbox{\epsfxsize=4.8cm \epsfbox{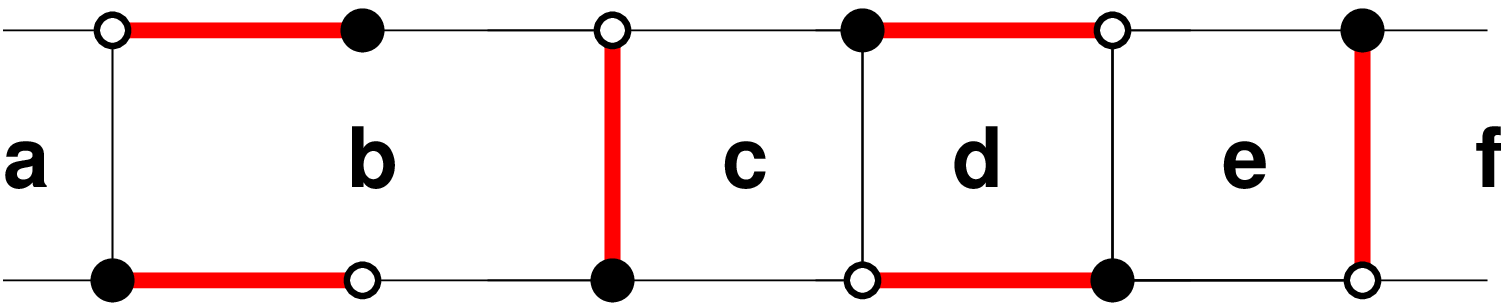}}} \\
a(b^\bullet)^{-1}(c^\bullet)^{-1}d^{-1}(e^\bullet)^{-1} & a b^{-1} (c^\bullet)^{-1}e^{-1}f & ab^{-1}(d^\bullet)^{-1}\\
& & \\
\raisebox{-.0cm}{\hbox{\epsfxsize=4.8cm \epsfbox{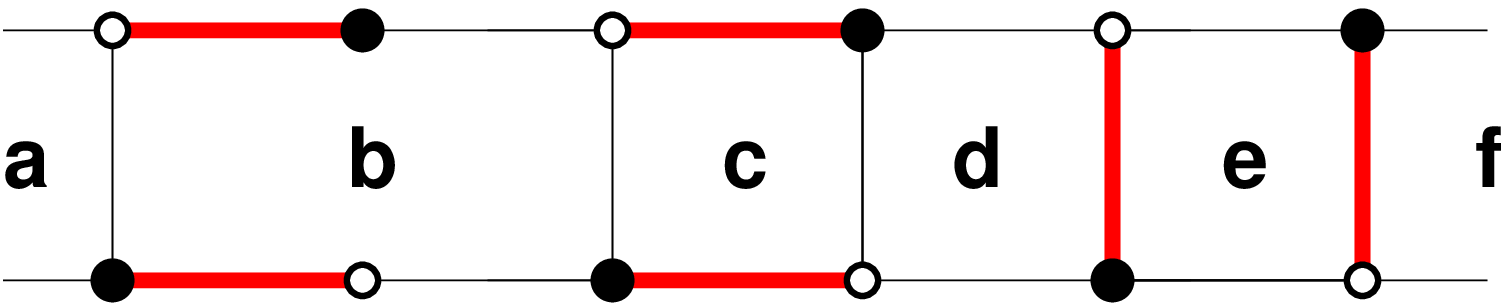}}} & \raisebox{-.0cm}{\hbox{\epsfxsize=4.8cm \epsfbox{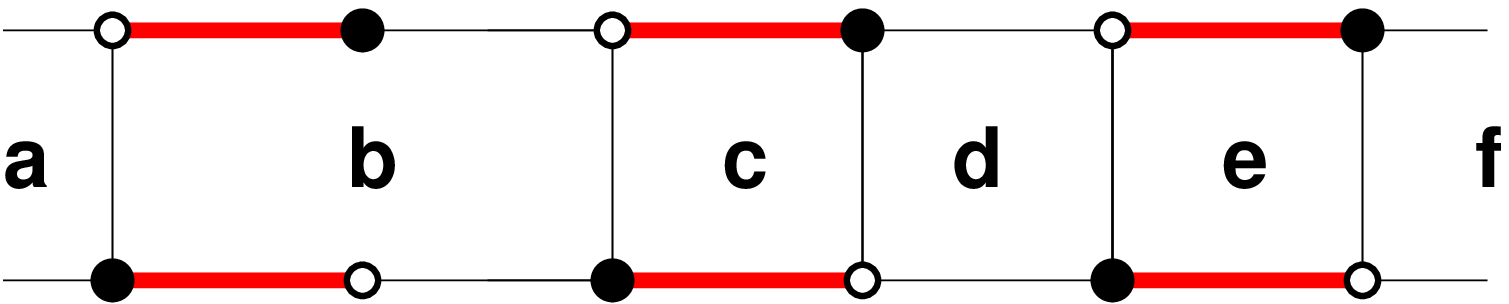}}}& \\
a c^{-1}(e^\bullet)^{-1} & ac^{-1}d e^{-1}f & \\
\end{matrix}
$$

\end{example}

\section{Conclusion and perspectives}

In this paper, we have introduced and solved a new non-commutative discrete integrable system of infinite dimension, and proved the positive Laurent property of its solutions in terms of initial data. In view of previous work in rank 2, it would be interesting to extend the definition of this system
to a full (infinite rank) non-commutative cluster algebra, allowing for more non-commutative transformations  preserving the Laurent property.

As a motivating example, let us consider the case of a general solution of the $A_1$ $T$-system,
for some fixed arbitrary initial data. Assume we have mutated the initial data so as to reach a situation where the new path is made locally of a succession of two up steps, say with  vertex values $a,b,c$ subject to $a^{-1}b=b^\bullet (a^\bullet)^{-1}$ and $b^{-1}c=c^\bullet (b^\bullet)^{-1}$.
We claim that the new ``mutation" $b\to b'$ at the central vertex (which takes us away from the $T$-system relations) still makes sense if defined by:
$$ b\, (b')^\bullet= a+c \quad {\rm or}\ {\rm equivalently} \quad b' b^\bullet=a^\bullet+c^\bullet$$
It is easy to see that this preserves the positive Laurent property, by noting that $(b')^\bullet=b^{-1}(a+c)=b^{-1}c +a^\bullet (b^\bullet)^{-1}$ is nothing but the conserved quantity $\Delta$ of Theorem \ref{consqthm}, expressed in terms of three aligned vertex values $a,b,c$. Indeed, this conserved quantity may be easily re-expressed in terms of the initial data, and it is readily seen that the expression for $(b)^\bullet$ divides that for $a+c$ on the left. From Remark \ref{selfad}, we deduce that $(b')^\bullet=b'$.
In addition, we learn that the new variable $b'$ satisfies the following ``triangle" relation:
$$ (a^\bullet)^{-1} (b')^\bullet c^{-1}= (c^\bullet)^{-1} b' a^{-1} 
 \quad {\rm or}\ {\rm equivalently} \quad  b' a^{-1} c= c^\bullet (a^\bullet )^{-1} (b')^\bullet$$
easily derived by straightforward algebra.
This relation generalizes the relations \eqref{quasicoinit}. It is very reminiscent of the triangular relations imposed in the 
finite rank case of non-commutative triangulations defined by Berenstein and Retakh \cite{RFO}.

Like in the classical case of \cite{DFK13}, we may try to extend the definition of the $A_1$ $T$-system 
for various geometries (half-plane, strip, etc.), in such a way as to preserve the positive Laurent property. 
We will address this problem in a later publication. We may also hope for a non-commutative version of 
Zamolodchikov's periodicity conjecture (say for type A).

Another direction of generalization should be to higher rank $T$- or $Q$-systems \cite{DFK10}. 
Many questions remain unanswered, such as the interpretation of higher order quasideterminants in 
terms of paths or dimers with non-commutative weights.

\end{document}